\documentclass[10pt,twoside]{article}
\usepackage[hmarginratio=1:1,a4paper,text={16cm,23cm}]{geometry}
\usepackage[english]{babel}
\usepackage[utf8]{inputenc}
\usepackage[T1]{fontenc}
\usepackage{amsmath,amsfonts,amssymb,amsthm}
\usepackage[dvipsnames]{xcolor}
\usepackage{tikz}
	\usetikzlibrary{shapes,arrows,calc,positioning,decorations.text,decorations.pathreplacing}
\usepackage{dsfont,color,graphicx,mathrsfs,enumerate,pdfpages,setspace}
\usepackage{caption}
\usepackage{url}
\usepackage[backref=page]{hyperref}
	\hypersetup{allcolors=blue,colorlinks=true,breaklinks=true,bookmarksopen=true} 
\usepackage{fancyhdr}
\tikzset{
blockR/.style={
  draw, 
  rectangle, 
  minimum height=1.5cm, 
  minimum width=3cm, align=center
  },
blockC/.style={
  draw, 
  circle, 
  minimum height=1.2cm, 
  align=center
  }
}
\makeatletter
\renewcommand{\thefigure}{\ifnum \c@section>\z@ \thesection.\fi
 \@arabic\c@figure}
\@addtoreset{figure}{section}
\makeatother
\newcounter{count}[section]\numberwithin{count}{section}
\numberwithin{equation}{section}

\newtheorem{theorem}[count]{Theorem}

\newtheorem{proposition}[count]{Proposition}
\newtheorem{corollary}[count]{Corollary}

\newtheorem{lemma}[count]{Lemma}
\newtheorem{assumption}[count]{Assumption}
\newenvironment{remark}[1][]{
	\par\noindent\refstepcounter{count}\textbf{Remark}
	\ifx\newenvironment#1\newenvironment
		\textbf{\arabic{section}.\arabic{count}.}
	\else
		\textbf{\arabic{section}.\arabic{count}} (#1)\textbf{.}
	\fi
}{
	\leavevmode\unskip\penalty9999 \hbox{}\nobreak\hfill\quad\hbox{{$\diamondsuit$}}
}
\newenvironment{acknowledgements}{\noindent\textbf{Acknowledgements: }}{}
\renewenvironment{abstract}[3][]{
	\def\abstracttoc{#1}
	\def\abstractkeywords{#2}
	\def\abstractmsc{#3}
	\begin{center}\begin{minipage}[c]{.8\textwidth}
	\textbf{Abstract: }
}{
	\abstracttoc\vspace{2ex}
	
	\noindent\textbf{Keywords: }\abstractkeywords
	\\\noindent\textbf{MSC 2010: }\abstractmsc
	\end{minipage}\end{center}
}
\makeatletter
\renewenvironment{proof}[1][\proofname]{\par
  \pushQED{\qed}%
  \normalfont \topsep6\p@\@plus6\p@\relax
  \trivlist
  \item[\hskip\labelsep
        \bfseries
    #1\@addpunct{\scantokens{:}}]\ignorespaces
}{%
  \popQED\endtrivlist\@endpefalse
}
\makeatother

\fancypagestyle{plain}{
	\fancyhf{}
	
}
\fancypagestyle{main}{
	\fancyhf{}
	\fancyhead[co]{\runtitle}
	\fancyhead[ce]{\runauthor}
	\fancyfoot[c]{\thepage}
	
}

\definecolor{darkred}{rgb}{0.9,0.1,0.1}

\renewcommand{\tilde}[1]{\widetilde{#1}}

\newcommand*{\B}{\mathbb{B}}

\DeclareMathOperator{\e}{e}
\newcommand*{\E}{\mathbb{E}}

\newcommand*{\Leb}{\mathbb{L}}

\renewcommand*{\P}{\mathbb{P}}

\newcommand*{\R}{\mathbb{R}}
\DeclareMathOperator{\Id}{Id}

\DeclareMathOperator{\Vect}{Vect}

\newcommand*{\indic}{\mathds{1}}

\DeclareMathOperator{\diag}{diag}

\newcommand*{\I}{\mathbf{I}}
\newcommand*{\X}{\mathbf{X}}
\newcommand*{\Y}{\mathbf{Y}}
\newcommand*{\Z}{\mathbf{Z}}
\newcommand*{\Dir}{\mathscr{D}}%\DeclareMathOperator{\Dir}{Dir}

\begin{document}
\title{Fluctuations of the Empirical Measure of\\Freezing Markov Chains}
\def\runtitle{Fluctuations of the Empirical Measure of Freezing Markov Chains}

\author{
	Florian \textsc{Bouguet}
		%\thanks{First footnote}
		\and
	Bertrand \textsc{Cloez}
		%\thanks{Second footnote}
}
\def\runauthor{
	Florian \textsc{Bouguet}, Bertrand \textsc{Cloez}
}

\date{
    \emph{Inria Nancy -- Grand Est, BIGS, IECL}\\
    \emph{MISTEA, INRA, Montpellier SupAgro, Univ. Montpellier}\\[2ex]
    %\today\\\textbf{\texttt{- \!\!- work in progress - \!\!-}}
}

\pagestyle{main}
\maketitle
\begin{abstract}[
	\tableofcontents%displays the ToC in the abstract
]{
	Markov chain; Long-time behavior; Piecewise-deterministic Markov process; Ornstein-Uhlenbeck process; Asymptotic pseudotrajectory
}{
	60J10; 60J25; 60F05
}
	In this work, we consider a finite-state inhomogeneous-time Markov chain whose probabilities of transition from one state to another tend to decrease over time. This can be seen as a cooling of the dynamics of an underlying Markov chain. We are interested in the long time behavior of the empirical measure of this freezing Markov chain. Some recent papers provide almost sure convergence and convergence in distribution in the case of the freezing speed $1/n^\theta$, with different limits depending on $\theta<1,\theta=1$ or $\theta>1$. Using stochastic approximation techniques, we generalize these results for any freezing speed, and we obtain a better characterization of the limit distribution as well as rates of convergence as well as functional convergence.
\end{abstract}

\section{Introduction}
\label{sec:Intro}
Let $(i_n)_{n \geq 1}$ be an inhomogeneous-time Markov chain with state space $\{1,\dots,D\}$ with the following transitions when $i\neq j$:
\[\P\left(i_{n+1}=j|i_n=i\right)=q_n(i,j),\quad q_n(i,j)= p_n (q(i,j) + r_n(i,j)),\]
where $(p_n)_{n\geq 1}$ is a decreasing sequence converging toward some $p\in [0,1]$, the remainders $r_n(i,j)$ tend to $0$ (fast enough) and $q$ is the discrete generator of some $\{1,\dots,D\}$-valued ergodic Markov chain. This model is related to the simulated annealing algorithm, and the sequence $(p_n)_{n\geq 1}$ can be interpreted as the cooling scheme of an underlying Markov chain generated by $q$. If $p<1$, since $\lim_{n\to+\infty}q_n(i,j)=pq(i,j)$, the probability of $(i_n)_{n\geq1}$ to move decreases over time, from which the appellation \emph{freezing Markov chain}.

The behavior of $(i_n)_{n \geq 1}$ is simple enough to understand, and depends on the summability of the sequence $(p_n)_{n\geq1}$. The chain $(i_n)_{n \geq 1}$ shall converge in distribution to the unique invariant probability $\nu^\top$ associated to $q$ if $\sum_{n=1}^\infty p_n = + \infty$ (see Theorem~\ref{thm:CVMarkovChain} below). On the other hand, if $\sum_{n= 1}^\infty p_n <+\infty$, the Markov chain shall freeze along the way, as a consequence of the Borel-Cantelli Lemma. Then, we shall assume that $\sum_{n=1}^\infty p_n = + \infty$, so that we can investigate the convergence of the empirical distribution $x_n^\top=\frac{1}{n} \sum_{k=1}^n \delta_{i_k}$.

The problem of the convergence of this empirical measure can be traced back to the thesis of Dobru\v{s}in \cite{Dob53}, and several questions are still open, as pointed out in the recent article \cite{EV16}. Some results can be obtained from the general theory developed in \cite{SV05,Pel12}, and \cite{DS07,EV16} study the present model. In particular, convergence results are only obtained for a freezing rate of the form $p_n=a/n^\theta$ (and $r_n(i,j)=0$). More precisely,
\begin{itemize}
\item if $\theta<1$ then $(x_n)_{n\geq1}$ converges to $\nu$ in probability; see \cite[Theorem~1.2]{DS07}.
\item if $\theta<1/2$, then $(x_n)_{n\geq1}$ converges to $\nu$ a.s. This can be extended to $1/2\leq\theta<1$ when the state space contains only two points; see \cite[Theorem~1.2]{DS07} and \cite[Corollary~2]{EV16}.
\item if $\theta<1$ and $D=2$, then, up to an appropriate scaling, the empirical measure $(x_n)_{n\geq1}$ converges in distribution to a Gaussian distribution; see \cite[Theorem~2]{EV16}.
\item if $\theta=1$ then $(x_n)_{n\geq1}$ converges in distribution, and the moments of the limit probability are explicit. If $q$ corresponds to the complete graph (see Section~\ref{sec:Applications}) then this limit probability is the Dirichlet distribution. When $D=2$, this covers classical distribution such as Beta, uniform, Arcsine or Wigner distributions; see \cite[Theorems~1.3 and 1.4]{DS07} and \cite[Theorem~1]{EV16}.
	\item when $D=2$, some convergence results are established for $(x_n)_{n\geq1}$ for general sequences $(p_n)_{n\geq 1}$, under technical conditions that we find hard to check in practice; see \cite[Theorem~3]{EV16}.
 \end{itemize}

We shall refer to the case $\theta<1$ as \emph{standard}, since it is related to classic laws of large numbers and central limit theorems. This case was called \emph{subcritical} in \cite{EV16}, in comparison with the \emph{critical} case $\theta=1$. Since we can slightly generalize this critical case here, the term \emph{non-standard} will be preferred from now on. In the present article, we generalize the aforementioned results by proving that, in the standard case, if $\sum_{n=1}^\infty (p_n n^2)^{-1}<+ \infty$ then $(x_n)_{n\geq1}$ converges to $\nu$ a.s., and we also give weaker conditions for convergence in probability; this is the purpose of Theorem~\ref{thm:StandardAS}. Under slightly stronger assumptions and up to a rescaling, we obtain convergence of $(x_n)_{n\geq1}$ to a Gaussian distribution with explicit variance in Theorem~\ref{thm:Standard}. Finally, if $p_n \sim a/n$, then $(x_n)_{n\geq1}$ converges in distribution exponentially fast to a limit probability (see Theorem~\ref{thm:NonStandard}). This distribution is characterized as the stationary measure of a piecewise-deterministic Markov process (PDMP), possesses a density with respect to the Lebesgue measure and satisfies a system of transport equations; see Propositions~\ref{proposition:ErgodicityEZZ} and \ref{prop:Pi_PDE}. Furthermore, Corollary~\ref{coro:PitoNor} links the standard and non-standard setting by providing a convergence of the rescaled stationary measure of the PDMP to a Gaussian distribution as the switching accelerates. We also investigate the complete graph dynamics in Section~\ref{sec:Applications} and are able to derive explicit results in Propositions~\ref{prop:DirichletEZZ} and \ref{prop:DirichletOU}. Most of our convergence results are also provided with quantitative speeds and functional convergences.

In contrast with the P\'{o}lya Urns model (see for instance \cite{Gou97}), all these results of convergences in distribution are not almost sure. However, note that, by letting $p_n=1$ for all $n\geq 1$, we can recover classical limit theorems for homogeneous-time Markov chains (see \cite{Jon04}). Furthermore, the remainder term $r_n(i,j)$ enables us to deal with different freezing schemes (see Remark~\ref{rk:remainder}). In particular, the proofs in \cite{DS07} and \cite{EV16} are mainly based on the method of moments, which is why more stringent assumptions are considered there. Our approach is completely different, and is based on the theory of asymptotic pseudotrajectories detailed in \cite{Ben99} and revisited in \cite{BBC17}.

Briefly, a sequence is an asymptotic pseudotrajectory of a flow if, for any given time window, the sequence and the flow starting from the same point evolve close to each other (see for instance \cite{BH96,Ben99}). This definition can be formalized for dynamical systems and be extended to discrete sequences of probabilities and continuous Markov semi-groups. This theory allows us to derive the behavior of the sequence of empirical measures $(x_n)_{n\geq1}$ from the one of auxiliary continuous-time Markov processes. The interested reader may find illustrations of this phenomenon in \cite[Figures~3.1, 3.2 and 3.3]{BBC17}, see also Figure~\ref{fig:interpolation}. In the present paper, depending on whether we work in a standard or non-standard setting, these processes are either a diffusive process or a switching PDMP. The careful study of these limit processes is of interest \emph{per se}, and is done in Section~\ref{sec:MarkovProcesses}. More precisely, Gaussian distributions appear naturally since we deal with an Ornstein-Uhlenbeck process generated by
\begin{equation}
\mathcal L_\text{O}f(y)=-y\cdot\nabla f(y)+\nabla f(y)^\top\Sigma^{(p,\Upsilon)}\nabla f(y),
\label{eq:GeneratorOU}
\end{equation}
where $\Sigma^{(p,\Upsilon)}$ is a $D\times D$ real-valued matrix such that
\begin{equation}
\Sigma_{k,l}^{(p,\Upsilon)}= \frac1{1+\Upsilon}\sum_{i=1}^D \nu_i \left[\sum_{j=1}^D q(i,j) (h_{l,j} - h_{l,i})(h_{k,j} - h_{k,i)} ) - p \left( \nu_k-\indic_{i=k} \right) \left(\nu_l-\indic_{i=l}\right)\right] ,\label{eq:DefSigma}
\end{equation}
with $p$ and $h$ respectively defined in Assumption~\ref{assumption:FreezingSpeed}, and in \eqref{eq:DefH}.
On the contrary, we shall see that, in a non-standard framework, the empirical measure is linked to a PDMP, called \emph{exponential zig-zag process}, generated by
\begin{equation}
\mathcal L_\text{Z}f(x,i)=(e_i-x)\cdot\nabla_xf(x,i)+\sum_{j\neq i}aq(i,j)[f(x,j)-f(x,i)].
\label{eq:GeneratorEZZ}
\end{equation}
These Markov processes shall be defined and studied more rigorously in Section~\ref{sec:MarkovProcesses}. In particular, besides some classic long-time properties (regularity, invariant measure, rate of convergence\dots), we prove in Theorem~\ref{sec:EZZtoOU} the convergence of the exponential zig-zag process to the Ornstein-Uhlenbeck process when the frequency of jumps accelerates, i.e. when $a\to+\infty$.

The rest of this paper is organized as follows. In Section~\ref{sec:MainResults}, we specifiy the notation and assumptions mentioned earlier, that will be used in the whole paper. We also state convergence results for $(x_n)_{n\geq1}$, which are Theorems~\ref{thm:NonStandard}, \ref{thm:StandardAS} and \ref{thm:Standard}. We study the long-time behavior of the two auxiliary Markov processes in Section~\ref{sec:MarkovProcesses} and investigate the case of the complete graph in Section~\ref{sec:Applications}, for which it is possible to get explicit formulas. The paper is then concluded with the proofs of the main theorems in Section~\ref{sec:proofs}.

\clearpage
\section{Freezing Markov chains}
\label{sec:MainResults}

\subsection{Notation}
\label{sec:Notation}
We shall use the following notation throughout the paper:
\begin{itemize}
	\item If $d$ is a positive integer, a multi-index is a $d$-tuple $N=(N_1,\dots,N_d)\in (\{0,1,\dots\}\cup\{+\infty\})^d$; the set of multi-indices is endowed with the order $N\leq\widetilde N$ if, for all $1\leq i\leq d,N_i\leq\widetilde N_i$. We define $|N|=\sum_{i=1}^dN_i$ and and we identify an integer $N$ with the multi-index $(N,\dots,N)$. Likewise, for any $x\in\R^d$, let $|x|=\sum_{i=1}^d|x_i|.$
	\item For some multi-index $N$ and an open set $U\subseteq\R^d, \mathscr C^N(U)$ is the set of functions $f:U\to\R$ which are $N_i$ times continuously differentiable in the direction $i$. For any $f\in\mathscr C^N(U),$ we define
	\[f^{(N)}=\partial_{1}^{N_1}\dots\partial_{d}^{N_d}f,\quad\|f^{(N)}\|_\infty=\sup_{x\in\R^d}|f^{(N)}(x)|.\]
	When there is no ambiguity, we write $\mathscr C^N$ instead of $\mathscr C^N(U)$, and denote by $\mathscr C^N_b$ and $\mathscr C^N_c$ the respective sets of bounded $\mathscr C^N$ functions and of compactly supported $\mathscr C^N$ functions.
	\item Let $\triangle$ be the simplex of $\R^D$ defined by
	\[\triangle=\left\{(x_1,\dots,x_D)\in\R^D:x_i\in[0,1],\sum_{i=1}^Dx_i=1\right\},\]
	and $E=\triangle\times\{1,\dots,D\}$.
	\item We denote by $\mathscr L(X)$ the probability distribution of a random vector $X$, and we identify the measures over $\{1,\dots,D\}$ with the $1\times D$ real-valued matrices. Let $\Leb$ be the Lebesgue measure over $\R^D$.
	\item If $\mu,\nu$ are probability measures and $f$ is a function, we write $\mu(f)=\int f(x)\mu(dx)$. For a class of functions $\mathscr F$, we define
	\[d_\mathscr{F}(\mu,\nu)=\sup_{f\in\mathscr F}|\mu(f)-\nu(f)|.\]
	Note that, for every class of functions $\mathscr F$ considered in this paper, convergence in $d_{\mathscr F}$ implies (and is often equivalent to) convergence in distribution (see \cite[Lemma~5.1]{BBC17}). In particular, let
	\[W(\mu,\nu)=\sup_{|f(x)-f(y)|\leq|x-y|}|\mu(f)-\nu(f)|,\quad d_{\text{TV}}(\mu,\nu)=\sup_{\|f\|_\infty\leq1}|\mu(f)-\nu(f)|\]
	be respectively the Wasserstein distance and the total variation distance.
	\item For $\theta\in(0,+\infty)^D,$ let $\Dir(\theta)$ be Dirichlet distribution over $\R^D$, i.e. the probability distribution with probability density function
	\[x\mapsto\frac{\Gamma\left(\sum_{k=1}^D\theta_k\right)}{\prod_{k=1}^D\Gamma(\theta_k)}\prod_{k=1}^Dx_k^{\theta_k-1}\indic_{\{x\in\triangle\}}.\]
	For $\theta_1,\theta_2>0,$ let $\beta(\theta_1,\theta_2)$ be the Beta distribution over $\R$, i.e. the probability distribution with probability density function
	\[x\mapsto\frac{\Gamma(\theta_1+\theta_2)}{\Gamma(\theta_1)\Gamma(\theta_2)}x^{\theta_1}(1-x)^{\theta_2}\indic_{0<x<1}.\]
	\item Let $x\wedge y:=\min(x,y)$ and $x\vee y:=\max(x,y)$ for any $x,y\in\R$.
	\item We write, for $n\geq1,u_n=O(v_n)$ if there exists some bounded sequence $(h_n)_{n\geq1}$ such that $u_n=h_nv_n$. Moreover, if $\lim_{n\to+\infty}h_n=0$, then we write $u_n=o(v_n)$.
\end{itemize}

\subsection{Assumptions and main results}
\label{sec:AssumptionsMainResults}
Let $D$ be a positive integer and $(i_n)_{n\geq1}$ be a $\{1,\dots,D\}$-valued inhomogeneous-time Markov chain such that, $\forall i\neq j$,
\[\P\left(i_{n+1}=j|i_n=i\right)=q_n(i,j).\]
The following assumption, which will be in force in the rest of the paper, describes the behavior of the transitions $q_n$ as time goes by.

\begin{assumption}[Freezing speed]
\label{assumption:FreezingSpeed}
Assume that that the matrix $\Id+q$ is irreducible and, for $n\geq1$ and $i\neq j$,
\begin{equation}
q_n(i,j)=p_n\left(q(i,j)+r_n(i,j)\right),
\label{eq:FreezingSpeed}
\end{equation}
where $(p_n)$ is a sequence decreasing to $p\in[0,1]$ such that $\sum_{n=1}^\infty p_n = + \infty$, and $\lim_{n\to+\infty}r_n(i,j)=0$. For $i\neq j$, assume $q(i,j)\geq0,q_n(i,j)\geq0$ and
\[q(i,i)=-\sum_{j\neq i}q(i,j),\quad q_n(i,i)=-\sum_{j\neq i}q_n(i,j).\]
\end{assumption}

Note that we do not require $(p_n)_{n\geq1}$ to converge to 0. Of course, if $p>0$, then the series $\sum_np_n$ trivially diverges; as pointed out in the introduction, if this series converge then the problem is trivial. In fact, if $p_n=1$ and $r_n(i,j)=0$ for any integers $i,j,n$, then the freezing Markov chain $(i_n)_{n\geq1}$ is a classic Markov chain. When $p=0$, the dynamics of Assumption~\ref{assumption:FreezingSpeed} corresponds to the lazier and lazier random walk introduced in \cite{BBC17}.

\begin{remark}[Irreducibility or indecomposability]
\label{rk:indecomposability}
The irreducibility of the transition matrix $\Id+q$ associated to $q$ is a classic hypothesis when it comes to Markov chains, since otherwise we can split their state space into different recurrent classes. However, the result of the present article can be extended to \emph{indecomposable}\footnote{The algebric term \emph{indecomposable} also exists for matrices, and is sometimes mistaken for irreducibility. Throughout this paper, a Markov chain (or its associated transistion matrix) is said \emph{indecomposable} if it admits a unique recurrent class.} Markov chains, which is a weaker concept. For instance, the transition matrix
\[\left[\begin{array}{ccc}
0&1&0\\
1&0&0\\
1/3&1/3&1/3
\end{array}\right]\]
is indecomposable but not irreducible. Namely, $\Id+q$ is irreducible if it cannot be written as
\[P^\top(\Id+q)P=\left[\begin{array}{ll}
A&0\\B&A'
\end{array}\right],\] 
where $A,A'$ are square matrices and $P$ is a permutation matrix. We could allow such a decomposition, as long as $B$ has a nonzero entry.

In any case, $\Id+q$ possesses a unique absorbing class of states on which it is irreducible. Using Perron-Frobenius Theorem (see \cite[Theorem~2p.53]{Gan59}), the matrix $\Id+q$ possesses a unique invariant measure $\nu^\top$, and the associated chain converges toward it under aperiodicity assumptions (see also Remark~\ref{re:ErgodicityIndecomposability}). Note that aperiodicity hypotheses are not relevant for the freezing Markov chain whenever $p<1$, since the freezing scheme automatically provides aperiodicity to the Markov chain.
\end{remark}

Under Assumption~\ref{assumption:FreezingSpeed}, $\Id+q$ possesses a unique invariant distribution $\nu^\top$, which writes $\nu^\top q=0$; let $\nu\in\triangle$ be its associated vector.

\begin{remark}[Interpretation of the term $r_n(i,j)$]
\label{rk:remainder}
The remainder $r_n(i,j)$ in \eqref{eq:FreezingSpeed} can either model small perturbations of the main freezing speed $p_nq(i,j)$, or a multiscale freezing scheme with $p_n$ being the slowest freezing speed. For instance, the case
\[q_n=\left[
\begin{array}{cc}
-n^{-\theta}&n^{-\theta}\\
n^{-(\theta+\tilde\theta)}&-n^{-(\theta+\tilde\theta)}
\end{array}
\right],\quad \theta,\tilde\theta>0\]
is covered by Assumption~\ref{assumption:FreezingSpeed}, with
\[q=\left[
\begin{array}{cc}
-1&1\\
0&0
\end{array}
\right],\quad p_n=n^{-\theta}.\]
\end{remark}

The following result characterizes the long-time behavior of the inhomogeneous Markov chain $(i_n)_{n\geq1}$.

\begin{theorem}[Convergence of the freezing Markov chain]
\label{thm:CVMarkovChain}
Under Assumption~\ref{assumption:FreezingSpeed}, if either $p<1$, or $p=1$ and $\Id+q$ is aperiodic,
\[\lim_{n\to+\infty}i_n=\nu^\top\text{ in distribution}.\]
\end{theorem}

Now, let us define $(e_1,\dots,e_D)$ the natural basis of $\R^D$ and introduce two different scaling rates
\begin{equation}
\gamma_n=\frac1n,\quad\alpha_n=\sqrt{\frac{p_n}{\gamma_n}},
\label{eq:DefGammaAlpha}
\end{equation}
and the associated rescaled vectors
\begin{equation}
x_n=\gamma_n\sum_{k=1}^ne_{i_k},\quad y_n=\alpha_n(x_n-\nu).
\label{eq:DefXY}
\end{equation}
It is clear that \eqref{eq:DefXY} writes
\begin{equation}
x_{n+1}=\frac{\gamma_{n+1}}{\gamma_n}x_n+\gamma_{n+1}e_{i_{n+1}},
\label{eq:DefXAlgoSto}
\end{equation}
that the vector $x_n$ belongs to the simplex $\triangle$ and that $(x_n,i_n)\in E=\triangle\times\{1,\dots,D\}$. We highlight the fact that, in general, the sequence $(x_n)_{n\geq1}$ is not a Markov chain by itself, but $(x_n,i_n)_{n\geq1}$ is.

\begin{remark}[Interpretation of $\triangle$]
The transpose $x\mapsto x^\top$ is a natural bijection between $\triangle$ and the set of probability measures over $\{1,\dots,D\}$. Then, the sequence $(x_n^\top)_{n\geq1}$ can be viewed as the sequence of empirical measures of the Markov chain $(i_n)_{n\geq1}$. From that viewpoint, we highlight the fact that the $L^1$ norm over $\triangle$ can be interpreted (up to a multiplicative constant) as the total variation distance: indeed, for any $x,\tilde x\in\triangle,$ 
\[|x-\tilde x|=\frac12d_\text{TV}\left(x^\top,\tilde x^\top\right)=\frac12d_\text{TV}\left(\sum_{i=1}^D x_i\delta_i,\sum_{i=1}^D \tilde x_i\delta_i\right).\]
\end{remark}

\begin{remark}[Weighted means]
Note that one could consider weighted means of the form
\[x_n=\frac{1}{\sum_{k=1}^n \omega_k }\sum_{k=1}^n\omega_k e_{i_k},\]
for any sequence of positive weights $(\omega_n)_{n\geq1}$, as in \cite[Remark~1.1]{BC15} or \cite[Section~3.1]{BBC17}. Then, we define $\gamma_n=\sum_{k=1}^n\omega_k$, and Theorem~\ref{thm:StandardAS} below still holds with the bound
\[\left| x_n- \nu \right|\leq C\exp\left(-v\sum_{k=1}^n\gamma_k\right).\]
\end{remark}

Following \cite{Ben99,BBC17}, and given sequences $(\gamma_n)_{n\geq1},(\epsilon_n)_{n\geq1}$, we define the following parameter which rules the speed of convergence in the context of standard fluctuations:
\begin{equation}
\lambda(\gamma,\epsilon)=-\limsup_{n\to+\infty}\frac{\log(\gamma_n\vee \epsilon_n)}{\sum_{k=1}^n \gamma_k}.
\label{eq:DefLambdaGammaEpsilon}
\end{equation}

Finally, we need to introduce a fundamental tool in the study of the standard fluctuations: the matrix $h$, which is solution of the multidimensional Poisson equation
\begin{equation}
\sum_{j\neq i}q(i,j)(h_j-h_i)=\nu-e_i,\quad\text{or equivalently }\sum_{j\neq i}q(i,j)(h_{k,j}-h_{k,i})=\nu_k-\indic_{i=k}
\label{eq:DefH}
\end{equation}
for all $1\leq i,k\leq D$, where we denoted by $h_i$ the $i$-th column vector of the matrix $h$. This solution is classically defined by
\[h_i= -\int_0^{+\infty} \left(e_i^\top \e^{t(\Id+q)} - \nu^\top \right) dt.\]
With the help of Perron-Frobenius Theorem (see \cite[Theorem~2p.53]{Gan59}), it is easy to see that $h$ is well-defined.

Throughout the paper, we shall treat two different cases, which entail different limit behaviors for the fluctuations of $(x_n)_{n\geq1}$ or $(y_n)_{n\geq1}$. Each of these cases corresponds to one of the two following assumptions.

\begin{assumption}[Non-standard behavior]
\label{assumption:NonStandard}
Assume that
\[p_n\underset{n\to+\infty}\sim\frac an.\]
\end{assumption}

Note that, under Assumption~\ref{assumption:NonStandard}, the sequences $(\gamma_n)_{n\geq1}$ and $(p_n)_{n\geq1}$ are equivalent up to a multiplicative constant and the scaling $(\alpha_n)_{n\geq1}$ is trivial, hence we are not interested in the behavior of $(y_n)_{n\geq1}$.

\begin{assumption}[Standard behavior]
\label{assumption:Standard}
\begin{enumerate}[i)]
	\item Assume that
	\[\limsup_{n\to+\infty}\frac{\gamma_n}{p_n}=0.\]
	\item Assume that
\[\frac{p_{n+1}}{p_n} = 1 + \frac{\Upsilon}{n}+ o\left(\frac{1}{n}\right), \quad \lim_{n \to+\infty} \frac{R_n}{ \sqrt{p_n \gamma_n} } =0,\]
with $R_n= \sup_i \sum_{j\neq i} |r_n(i,j)|$.
\end{enumerate}
\end{assumption}

Now, we have all the tools needed to study the behavior of the empirical measure $(x_n)_{n\geq1}$.

\begin{theorem}[Non-standard fluctuations]
\label{thm:NonStandard}
Under Assumptions~\ref{assumption:FreezingSpeed} and \ref{assumption:NonStandard},
\[\lim_{n\to+\infty}(x_n,i_n)=\pi\text{ in distribution},\]
where $\pi$ is characterized in Propositions~\ref{proposition:ErgodicityEZZ} and \ref{prop:Pi_PDE}.

Moreover, if there exist positive constants $A\geq1,\theta\leq1$ such that
\[\max_{j\neq i}(|r_n(i,j)|)\leq \frac A{n^{\theta}},\]
then, denoting by $\rho$ the spectral gap of $\Id+q$, for any
\[u<\frac\theta{A+\theta\left(1+\frac1{a\rho}\right)},\]
there exist a class of functions $\mathscr F$ defined in \eqref{eq:defF} and a positive constant $C$ such that
\[d_{\mathscr F}\left(\mathscr L(x_n,i_n),\pi\right)\leq Cn^{-v}.\]
\end{theorem}

It should be noted that our approach for the study of the long-time behavior of $(x_n,i_n)_{n\geq 0}$ also provides functional convergence for some interpolated process $(X_t,I_t)_{t\geq0}$ defined in \eqref{eq:interpolatedprocess} (see Lemma~\ref{lem:APTNonStandard}, from which Theorem~\ref{thm:NonStandard} is a straightforward consequence). Moreover, note that the speed of convergence provided by Theorem~\ref{thm:NonStandard} writes, for any function $f:\triangle\times\{1,\dots,D\}\to\R$, two times differentiable in the first variable, there exists a constant $C_f$ such that
\[|\E[f(x_n,i_n)]-\pi(f)|\leq C_f n^{-v}.\]

\begin{remark}[Is it possible to generalize Assumption~\ref{assumption:NonStandard}?]
This remarks leans heavily on the proof of Theorem~\ref{thm:NonStandard} and may be omitted at first reading. It is interesting to wonder whether it is possible to obtain non-standard fluctuations for a more general freezing speed $(p_n)_{n\geq1}$. To that end, let us try to mimic the computations of the proof of Lemma~\ref{lem:APTNonStandard} with $(\tilde x_n,i_n)_{n\geq1}$ with
\[\tilde x_n=\tilde\gamma_n\sum_{k=1}^ne_{i_k},\]
for any vanishing sequences $(\gamma_n)_{n\geq1}$ and $(\tilde\gamma_n)_{n\geq1}$. Our method being based on asymptotic pseudotrajectories, the limit of the rescaled process of $(x_n,i_n)_{n\geq1}$ belongs to a certain class of PDMPs which can be attained if, and only if,
\begin{equation}
\lim_{n\to+\infty}\frac{p_n}{\gamma_n}=C_1,\quad \lim_{n\to+\infty}\frac{\tilde\gamma_n}{\gamma_n}=C_2,\quad \lim_{n\to+\infty}\frac{1}{\gamma_n}\left(\frac{\tilde\gamma_{n+1}}{\tilde\gamma_n}-1\right)=-C_3,
\label{eq:ConditionNonStandard}
\end{equation}
with $C_1,C_2,C_3>0$. Without loss of generality, one can choose $\gamma_n=\tilde\gamma_n$ and $C_2=C_3=1$. Then, the third term of \eqref{eq:ConditionNonStandard} entails $\gamma_n=(n+o(1))^{-1}$ as $n\to+\infty$, which in turn implies $p_n=C_1n^{-1}+o(n^{-1})$ when injected in the first term of \eqref{eq:ConditionNonStandard}.

Also, note that assuming $A<1$ or $\theta>1$ in Theorem~\ref{thm:NonStandard} would not provide better speeds of convergence, since one would obtain a speed of the form
\[u<\frac{\theta\wedge1}{A\vee1+\theta\wedge1\left(1+\frac1{a\rho}\right)}.\]
\iffalse
Note that $(\gamma_n)$ has to be decreasing for \cite[Theorem~2.6]{BBC17} to hold, hence we assumed that $(p_n)$ is decreasing as well. A careful reading of the proof of \cite[Theorem~2.6]{BBC17} would allow us to slightly generalize it by considering $\bar\gamma_n=\sup_{k\geq n}\gamma_k.$
\fi
\label{rk:ConditionNonStandard}
\end{remark}

\begin{theorem}[Standard convergence of the empirical measure]
\label{thm:StandardAS}
Under Assumptions~\ref{assumption:FreezingSpeed} and \ref{assumption:Standard}.i),
\[\lim_{n\to+\infty}x_n=\nu\text{ in probability},\]
or equivalently in $L^1$.

Moreover, if $\sum_{n=1}^\infty\gamma_n^2p_n^{-1} <+ \infty,$ then $\lim_{n\to+\infty}x_n=\nu$ a.s.

Moreover, if $\ell= \lambda(\gamma, \gamma/p)\wedge\lambda(\gamma, R) >0,$ then, for any $v<\ell$ there exists a (random) constant $C>0$ such that
\[\left| x_n- \nu \right|\leq \frac{C}{n^v}\text{ a.s.}\]
\end{theorem}

\begin{theorem}[Standard fluctuations]
\label{thm:Standard}
Under Assumptions~\ref{assumption:FreezingSpeed} and \ref{assumption:Standard}, $(y_n)_{n\geq1}$ converges in distribution to the Gaussian distribution $\mathscr N\left(0,\Sigma^{(p,\Upsilon)}\right)$
\end{theorem}

The precise proofs of the main results are deferred to Section~\ref{sec:proofs}. As pointed out in the introduction, our proofs of Theorems~\ref{thm:NonStandard} and \ref{thm:Standard} rely on comparing $(x_n)_{n\geq1}$ and $(y_n)_{n\geq1}$ with auxiliary continuous-time Markov processes, using the theory of asymptotic pseudotrajectories and the SDE method. Then, these discrete Markov chains will inherit some properties of the Markov processes that we shall prove in Section~\ref{sec:MarkovProcesses}. In particular, the results we use provide functional convergence of the rescaled interpolating processes to the auxiliary Markov processes (see \cite[Theorem~2.12]{BBC17} and \cite[Th\'{e}or\`{e}me~4.II.4]{Duf96}).

\begin{remark}[Examples of freezing rates]
\label{rk:StandardFreezingRates}
For the sake of simplicity, consider $r_n(i,j)=0$ for all $i,j,n$.  Assumption~\ref{assumption:Standard} covers sequences $(p_n)_{n\geq1}$ of the form $p_n=n^{-\theta}$ for any $0<\theta<1$, since $\gamma_n^2p_n^{-1}=n^{\theta-2}$. In this case, $\ell=\lambda(n^{-1},n^{\theta-1})= 1 - \theta>0$.

But we can also consider more exotic freezing rates, for instance $p_n =\log(n)^\zeta n^{-1},$ for some $\zeta\geq1$. Then, $\gamma_n^2p_n^{-1}=n^{-1} \log(n)^{-\zeta}$. If $\zeta>1$, then the series converges and $\ell=1$. Our results do not provide almost sure convergence in the case $\zeta=1$, however, but only convergence in probability.

\iffalse
Note that the case where $(p_n)_{n\geq0}$ converges to 0 slightly faster than $(n^{-1})_{n\geq0}$ is not covered by Assumption~\ref{assumption:NonStandard} nor Assumption~\ref{assumption:Standard}. If $p_n=o(n^{-1})$ and $\sum_{n=1}^\infty p_n = +\infty$, for instance if $p_n=(n\log(n))^{-1}$, techniques would be different, but we can provide some heuristics. Consider the piecewise-constant process $(X_t,I_t)_{t\geq0}$ defined in \eqref{eq:interpolatedprocess} and $\gamma_n=p_n$.

One can think as $pn=1/nln(n)$ for a typical example.  We do not treat this case in this paper because the techniques of proof will be different but we can give some heuristics. If we consider the piecewise constant process $(X_t)$ defined in (mettre un numero pour $X_t, I_t$ juste apres 5.2) with $\gamma_n=1/nln(n)$ then $(X_{t+T})_{t\geq 0}$ should converge to the continuous-time Markov chain with generator $q$ for the Skorohod topology. However it is not an asymptotic trajectory. This can be seen as the Markov process level. Indeed, although the Ornstein-Uhlenbeck dynamics appears from the standart case by letting $a\to \infty$, the limit of the PDMP when $a\to \infty$ should be the process $(I_t,I_t)_{t\geq 0}$ as usual for slow-fast limit [?].
\fi

It should be noted that assuming that $(p_n)_{n\geq1}$ is decreasing, $\lim_{n\to+\infty}p_n =0$ and $\sum_{n=1}^\infty p_n=+\infty$ do not imply in general that $p_{n+1}\sim p_n$. A slight modification of the proof shows that, if $p_{n+1}$ is not equivalent to $p_n$, we have to assume the existence of a sequence $(\beta_n)_{n\geq1}$ such that
\[\lim_{n\to+\infty} \frac{p_n}{\gamma_n^2 \beta_{n}^2} \left( \frac{\beta_{n}}{\beta_{n-1}}(1-\gamma_n) - 1 \right) =-1,\quad \sum_{n =1}^\infty \frac{\gamma_n^2 \beta_{n}^2}{p_n}= + \infty,\quad\lim_{n\to+\infty} \beta_n\gamma_n =-1\]
and such that the sequences $(\gamma_n^2 \beta_{n}^2p_n^{-1})_{n \geq 1}$ and $(\beta_{n} \gamma_n)_{n\geq1}$ are decreasing; then the conclusion of Theorem~\ref{thm:Standard} holds.
\end{remark}

\section{The auxiliary Markov processes}
\label{sec:MarkovProcesses}
In this section, we study the ergodicity of the processes arising as limits of the freezing Markov from Section~\ref{sec:MainResults}. We also study their invariant measure, and provide explicit formulas when it is possible.

\subsection{The exponential zig-zag process}
\label{sec:EZZprocess}
In this section, we investigate the asymptotic properties of the exponential zig-zag process, which arise from the non-standard scaling of the Markov chain $(i_n)_{n\geq1}$. To this end, let $(\X_t,\I_t)_{t\geq0}$ be the strong solution of the following SDE (see \cite{IW89}), with values in $E$:
\begin{equation}
(\X_t,\I_t)=(\X_0,\I_0)+\int_0^t \left(A(\X_{s^-},\I_{s^-})+e_{\I_{s^-}}\right)ds+\sum_{j=1}^D\int_0^t B_{\I_{s^-},j}(\X_{s^-},\I_{s^-})N_{\I_{s^-},j}(ds),
\label{eq:SdeEZZ}
\end{equation}
where the $N_{i,j}$ are independent Poisson processes of intensity $aq(i,j)\indic_{\{i\neq j\}}$ and
\begin{equation}
A=\left[
\begin{array}{cccc}
-1&0&\cdots&0\\
0&\ddots&\ddots&\vdots\\
\vdots&\ddots&-1&0\\
0&\cdots&0&0
\end{array}
\right],\qquad B_{i,j}\left[
\begin{array}{cccc}
&&&0\\
&(0)&&\vdots\\
&&&0\\
0&\cdots&0&i-j
\end{array}
\right].
\label{eq:SdeEzz2}
\end{equation}
Thus, the infinitesimal generator of this process is $\mathcal L_\text{Z}$ defined in \eqref{eq:GeneratorEZZ} (see e.g. \cite{EK86,Dav93,Kol11}). Actually, the exponential zig-zag process is a PDMP; the interested reader can consult \cite{Dav93,BLBMZ15} for a detailed construction of the process $(\X,\I)$. Let us describe briefly its dynamics: setting $\I_0=i$, the process possesses a continuous component $\X$ which is exponentially attracted to the vector $e_i$. The discrete component $\I_t$  is piecewise-constant, and jumps from $i$ to $j$ following the epochs of the processes $N_{i,j}$, which in turn leads the continuous component to be attracted to $e_j$ (see Figure~\ref{fig:EZZ} for sample paths of the exponential zig-zag process, and Figure~\ref{fig:EZZTurnover} for a typical path in the framework of Section~\ref{sec:Turnover}).

The following result might be seen as a direct consequence of \cite[Theorem~1.10]{BLBMZ12} or \cite[Theorem~1.4]{CH15}, although these articles do not provide explicit rates of convergence, which are useful for instance in the proof of Corollary~\ref{coro:PitoNor}. 

\begin{figure}[htbp!]
\begin{center}
\includegraphics[width=0.329\textwidth]{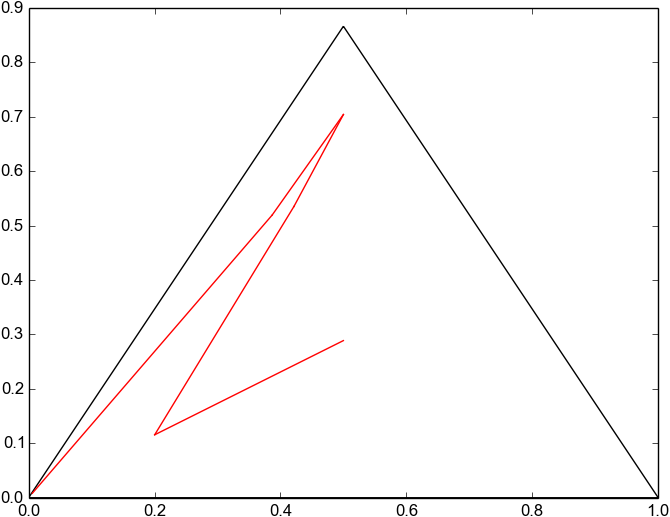}
\includegraphics[width=0.329\textwidth]{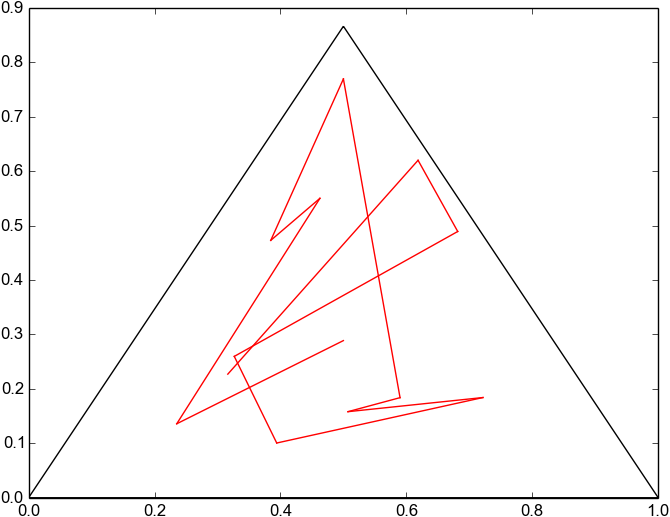}
\includegraphics[width=0.329\textwidth]{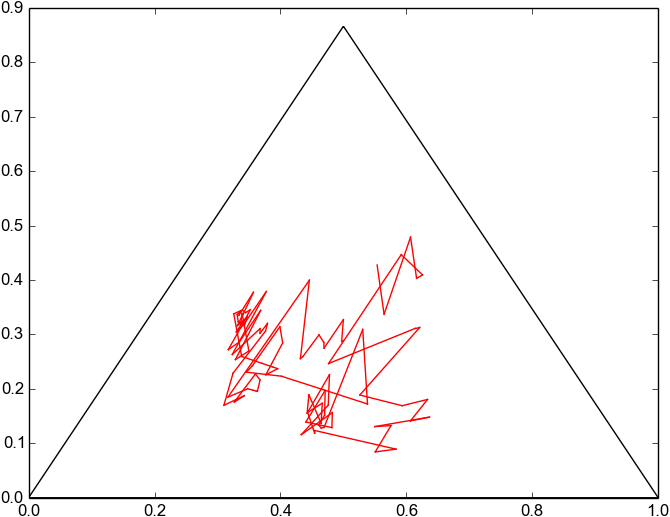}
\end{center}
\caption{Sample paths on $[0,5]$ of the exponential zig-zag process for $X_0=(1/3,1/3,1/3),q(i,j)=1$ and $a=0.5,a=2,a=20$ (from left to right).}
\label{fig:EZZ}\end{figure}

\begin{proposition}[Ergodicity]\label{proposition:ErgodicityEZZ}
The exponential zig-zag process $(\X_t,\I_t)_{t\geq0}$  admits a unique stationary distribution $\pi$. If $\rho$ is the spectral gap of $q$, then for any for any $v<a\rho (1+a\rho)^{-1}$, there exists a constant $C>0$ such that
\[\forall t\geq 0, \quad W\left((\X_t,\I_t),\pi\right)\leq C\e^{-vt}.\]

Moreover, if $\mathscr L(\I_0)=\nu^\top$, then
\[\forall t\geq 0, \quad  W\left((\X_t,\I_t),\pi\right)\leq W\left((\X_0,\I_0),\pi\right)\e^{-t}.\]
\end{proposition}

Note that the speed of convergence provided in Proposition~\ref{proposition:ErgodicityEZZ} can be improved when $D=2$, since we are able to use more refined couplings (see Proposition~\ref{proposition:ErgodicityEZZ2}).

\begin{proof}[Proof of Proposition~\ref{proposition:ErgodicityEZZ}]
The pattern of this proof follows \cite{BLBMZ12}. Let $(\X_t,\I_t, \tilde{\X}_t, \tilde{\I}_t)_{t\geq0}$ be the coupling for which the discrete components $\I$ and $\tilde \I$ are equal forever once they are equal once. Let $t>0$ and $\alpha\in(0,1)$. Firstly, note that, if $\I_{\alpha_t}=\tilde \I_{\alpha_t}$, then the processes always have common jumps and
\begin{equation}
|\X_t-\tilde \X_t|=|\X_{\alpha_t}-\tilde \X_{\alpha_t}|\e^{-t}\leq 2\e^{-t}.
\label{eq:proofErgodicityEZZD1}
\end{equation}
From the Perron-Frobenius theorem (see \cite{Gan59,Sal97}), for any $\varepsilon>0$, there exists $\tilde C>0$ such that
\[d_\text{TV}(\I_t,\tilde \I_t)\leq \tilde C\e^{-(a\rho-\varepsilon) t}.\]
Then there exists a coupling of the random variables $\I_{\alpha t}$ and $\tilde \I_{\alpha t}$ such that
\begin{equation}
\P(\I_{\alpha t}\neq\tilde \I_{\alpha t})\leq \tilde C\e^{-(a\rho-\varepsilon)\alpha t}.
\label{eq:proofErgodicityEZZD2}
\end{equation}
Now, combining \eqref{eq:proofErgodicityEZZD1} and \eqref{eq:proofErgodicityEZZD2},
\begin{align*}
\E\left[|(\X_t,\I_t)-(\tilde \X_t,\tilde \I_t)|\right]&\leq\E\left[\left.|(\X_t,\I_t)-(\tilde \X_t,\tilde \I_t)|\right|\I_{\alpha t}\neq\tilde \I_{\alpha t}\right]\P(\I_{\alpha t}\neq\tilde \I_{\alpha t})\\
&\quad+\E\left[\left.|(\X_t,\I_t)-(\tilde \X_t,\tilde \I_t)|\right|\I_{\alpha t}\neq\tilde \I_{\alpha t}\right]\P(\I_{\alpha t}\neq\tilde \I_{\alpha t})\\
&\leq 2\P(\I_{\alpha t}\neq\tilde \I_{\alpha t})+2\e^{-(1-\alpha)t}\\
&\leq 2\tilde C\e^{-(a\rho-\varepsilon)\alpha t}+2\e^{-(1-\alpha)t}.
\end{align*}
One can optimize this speed of convergence by taking $\alpha=(1+a\rho-\varepsilon)^{-1}$, and get
\begin{equation}
W\left((\X_t,\I_t),(\tilde\X_t,\tilde\I_t)\right)\leq C\e^{-vt}
\label{eq:proofErgodicityEZZD3}
\end{equation}
with $C=2\tilde C+2$ and $v=(a\rho-\varepsilon)(1+a\rho-\varepsilon)^{-1}$. Then, $(\mathscr L((\X_t,\I_t))$ is a Cauchy sequence and converges to a (stationary) distribution $\pi$. Letting $\mathscr L(\tilde\X_0,\tilde\I_0)=\pi$ in \eqref{eq:proofErgodicityEZZD3}, achieves the proof in the general case.

Now, if $\mathscr L(\I_0)=\nu^\top$, then $\mathscr L(\I_0)=\mathscr L(\tilde \I_0)$; we can let $\I_0=\tilde \I_0$, and then it suffices to use \eqref{eq:proofErgodicityEZZD1} with $\alpha=0$.
\end{proof}

If Assumption~\ref{assumption:FreezingSpeed} is in force, there exists a unique invariant measure $\pi$, which satisfies
$$
\int_E \mathcal{L}_\text{Z} f(x,i) \pi(dx,di)=0,
$$
for any function $f$ smooth enough. Now, let us establish the absolute continuity of this invariant distribution with respect to the Lebesgue measure $\Leb$.

\begin{lemma}[Absolute continuity of the exponential zig-zag process]
\label{lemma:DensityEZZ}
Let $K\subset \mathring{\triangle}$ be a compact set. There exist constants $t_0,c_0>0$ and a neighborhood $V$ of $K$ such that, for any $(x,i)\in E$ and for all $t\geq t_0$,
\begin{equation}
\P(\X_t\in \cdot,\I_t=j|\X_0=x,\I_0=i) \geq c_0 \Leb (\cdot \cap V(y)).
\label{eq:abscont-traj}
\end{equation}
\end{lemma}

\begin{remark}[When $\Id+q$ is only indecomposable]
\label{re:ErgodicityIndecomposability}
This remark echoes Remark~\ref{rk:indecomposability} and describes the behavior of the Markov chain $(x_n,i_n)_{n\geq1}$ when $\Id+q$ is reducible but indecomposable. In that case, Proposition~\ref{proposition:ErgodicityEZZ} holds as well. However, $\Id+q$ possesses a unique recurrent class which is strictly contained in $\{1,\dots,D\}$, the vector $\nu$ possesses at least one zero and belongs to the frontier of the simplex $\triangle$, and $\pi(\mathring\triangle)=0$. It is then impossible to obtain an equivalent to Proposition~\ref{proposition:ErgodicityEZZ} with a convergence in total variation; when $\Id+q$ is irreducible, this is possible using techniques inspired from \cite[Proposition~2.5]{BMPPR15}.

If $\Id+q$ is indecomposable, one can obtain equivalents of Lemma ~\ref{lemma:DensityEZZ} and Proposition~\ref{prop:Pi_PDE} below by replacing the Lebesgue measure $\Leb$ on $\R^D$ by the Lebesgue measure on the linear subspace spanned by the recurrent class of $\Id+q$.
\end{remark}

\begin{proof}[Proof of Lemma~\ref{lemma:DensityEZZ}]
The proof is mainly based on H\"ormander-type conditions for switching dynamical systems obtained in \cite{BH12,BLBMZ15}. Using the notation of \cite{BLBMZ15}, let $F^i: x\mapsto e_i-x$ and then, if $D\geq3$,
$$
\forall x \in \triangle, \quad \mathcal{G}_0(x) =\Vect\{F^i(x)-F^j(x): i\neq j\} = \Vect\{e_i-e_j: i\neq j\}= \mathbb {R}^D,
$$ 
where $\Vect A$ denotes the vector space spanned by $A\subseteq\R^D$. If $D=2,$ then $\mathcal G_1(x)=\R^2$. As a consequence, the strong bracket condition of \cite[Definition~4.3]{BLBMZ15} is satisfied. In particular, using \cite[Theorems~4.2 and 4.4]{BLBMZ15}, we have that, for every $x\in \triangle$, $y\in \mathring{\triangle}$, there exist $t_0(x),c_0(x)>0$ and open sets $U_0(x), V(x,y)$, such that for all $x_0 \in U_0(x),  i,j \in \{1,\dots, D\}, A\subseteq \triangle$ and $t>t_0(x)$,
\[\P(\X_t\in A,\I_t=j|\X_0=x_0,\I_0=i) \geq c_0(x) \Leb (A \cap V(x,y)).\]
Now, $\triangle = \cup_{x\in \triangle} U_0(x)$ and is compact, so there exist $x_1,...,x_n$ such that $\triangle = \cup_{k =1}^n U_0(x_k)$ . In particular, setting $V(y)= \cup_{k =1}^n V(x_k,y)$, $c_0=\min_{1 \leq k \leq n} c_0(x_k)$, $t_0= \max_{1 \leq k \leq n} t_0 (x_k)$, we have, for all $x_0 \in \triangle,  i,j \in \{1,\dots, D\}, A\subseteq \triangle$ and $t>t_0$,
\[\P(\X_t\in A,\I_t=j|\X_0=x_0,\I_0=i) \geq c_0 \Leb (A \cap V(y)),\]
Once again, $K$ is compact so we can extract a finite family from the open sets $(V(y))_{y\in K}$. Using the Markov property, this holds for every $t\geq t_0$, which entails \eqref{eq:abscont-traj}.
\end{proof}

\begin{proposition}[System of transport equations for $\pi$]
\label{prop:Pi_PDE}
The distribution $\pi$ introduced in Proposition~\ref{proposition:ErgodicityEZZ} admits the following decomposition:
\begin{equation}
\pi=\sum_{i=1}^D\nu_i\pi_i\otimes\delta_i,\quad\pi_i(dx)=\varphi(x,i)dx.
\label{eq:DecompositionPi}
\end{equation}
where the function $\varphi$ satisfies, for any $(x,i)\in E$,
\begin{equation}
(D-1) \varphi(x,i)  + \sum_{k=1}^D x_k \partial_k \varphi(x,i) - \partial_i \varphi(x,i)  + \sum_{j=1}^D \frac{\nu_j}{\nu_i}a q(j,i) \varphi(x,j)=0.
\label{eq:PDE-proof}
\end{equation}
\end{proposition}

Once we will have proved that $\pi$ admits the decomposition \eqref{eq:DecompositionPi}, the next step is the characterization of $\varphi$. Indeed, since it satisfies
\begin{equation}
\label{eq:densite}
\sum_{i=1}^D\nu_i\int_\triangle \mathcal{L}_\text{Z} f(x,i) \varphi(x,i) dx=0,
\end{equation}
for every smooth enough function $f$, all we have to do is compute the adjoint operator of $\mathcal{L}_\text{Z}$. For general switching model, it would not possible to characterize $\varphi$ as a solution of a simple system of PDEs like \eqref{eq:PDE-proof}. However, the present form of the flow enables us to derive a simple expression for the adjoint operator of $\mathcal{L}_\text{Z}$. Before turning to the proof of Proposition~\ref{prop:Pi_PDE}, let us present the following formula of integration by parts over the simplex $\triangle$.

\begin{lemma}[Integration by parts over $\triangle$]
\label{lem:ipp}
For all $f,g\in\mathscr C_c^1\left(\mathring\triangle\right)$, and $k,l \in \{1,\dots,D\}$, we have
\[\int_{\triangle} g(x) (\partial_k - \partial_l) f(x) dx= -\int_{\triangle} (\partial_k-\partial_l) g(x)  f(x) dx.\]
\end{lemma}

\begin{proof}[Proof of Lemma~\ref{lem:ipp}]
Fix $l=1$ and let $\triangle_1=\left\{x_2,\dots,x_D\in[0,1]: \sum_{i=2}^D x_i \leq1\right\}$. Then,
\begin{align*}
&\int_{\triangle} g(x) \partial_k f(x) dx_1\dots dx_D
= \int_{\triangle_1} g\left(1-\sum_{i=2}^D x_i ,x_2,...\right) \partial_k f\left(1-\sum_{i=2}^D x_i ,x_2,...\right) dx\\
&= \int_{\triangle_1} \left[  \partial_k\left( g\left(1-\sum_{i=2}^D x_i ,x_2,...\right)  f\left(1-\sum_{i=2}^D x_i ,x_2,...\right) \right) \right. \\
&+ \partial_1 g\left(1-\sum_{i=2}^D x_i ,x_2,...\right)  f\left(1-\sum_{i=2}^D x_i ,x_2,...\right) + g\left(1-\sum_{i=2}^D x_i ,x_2,...\right)  \partial_1 f\left(1-\sum_{i=2}^D x_i ,x_2,...\right) \\
&\left. - \partial_k g\left(1-\sum_{i=2}^D x_i ,x_2,...\right)  f\left(1-\sum_{i=2}^D x_i ,x_2,...\right) \right] dx_1\dots dx_D.
\end{align*}
Now, as $g(0,x_2,\dots)=f(0,x_2,\dots)=0$ and $\partial_k 1=0$, use a (classic) multidimensional integration by parts to establish that
$$
\int_{\triangle_1} \partial_k\left( g\left(1-\sum_{i=2}^D x_i ,x_2,...\right)  f\left(1-\sum_{i=2}^D x_i ,x_2,...\right) \right) dx_1\dots dx_D= 0,
$$
which entails Lemma~\ref{lem:ipp}.
\end{proof}

\begin{proof}[Proof of Proposition~\ref{prop:Pi_PDE}]
Integrating \eqref{eq:abscont-traj} with respect to the unique invariant measure $\pi$, we obtain that $\pi$ admits an absolutely continuous part (note that uniqueness comes from Proposition~\ref{proposition:ErgodicityEZZ}). Since $\pi$ cannot have an absolutely continuous part and a singular one (see \cite[Theorem~6]{BH12}), $\pi$ admits a density with respect to the Lebesgue measure, which entails \eqref{eq:DecompositionPi}.

Now, let us characterize the function $\varphi$. We have
\begin{align*}
&\sum_{i,k=1}^D \nu_i \int_\triangle (-x_k + \mathbf{1}_{i=k}) \varphi(x,i) \partial_k f(x,i) dx \\
&\quad= \sum_{i=1}^D\left(-\sum_{k=1}^D  \nu_i \int_\triangle x_k \varphi(x,i) \partial_k f(x,i) dx+ \nu_i \int_\triangle  \varphi(x,i) \partial_i f(x,i) dx\right)
\end{align*}
and, using Lemma~\ref{lem:ipp}, for any $1\leq i\leq D$,
\begin{align*}
& -\sum_{k=1}^D \int_\triangle x_k \varphi(x,i) \partial_k f(x,i) dx
+ \int_\triangle  \varphi(x,i) \partial_i f(x,i) dx \\
&=-\sum_{k\neq i}  \int_\triangle x_k \varphi(x,i) \partial_k f(x,i) dx
 - \int_\triangle x_i \varphi(x,i) \partial_i f(x,i) dx+  \int_\triangle  \varphi(x,i) \partial_i f(x,i) dx\\
&=\sum_{k\neq i}\left[  \int_\triangle \partial_k\left(x_k \varphi(x,i)\right]  f(x,i) dx
 - \int_\triangle x_k \varphi(x,i) \partial_i f(x,i) dx
 -  \int_\triangle \partial_i \left(x_k \varphi(x,i) \right)  f(x,i) dx\right)\\
&\quad - \int_\triangle x_i \varphi(x,i) \partial_i f(x,i) dx + \int_\triangle  \varphi(x,i) \partial_i f(x,i) dx\\
&=\sum_{k\neq i} \left[\int_\triangle \partial_k\left(x_k \varphi(x,i)\right)  f(x,i) dx - \int_\triangle \partial_i \left(x_k \varphi(x,i) \right)  f(x,i) dx\right]\\
&=\sum_{k\neq i} \int_\triangle \partial_k\left(x_k \varphi(x,i)\right)  f(x,i) dx -   (1-x_i) \int_\triangle \partial_i \varphi(x,i)  f(x,i) dx.
\end{align*}
Hence, \eqref{eq:densite} writes
\begin{align*}
&\sum_{i=1}^D \int_\triangle \nu_if(x,i)\left[  \sum_{k\neq i}    \partial_k\left(x_k \varphi(x,i)\right)   -  (1-x_i) \partial_i \varphi(x,i)  + \sum_{j =1}^D \frac{\nu_j}{\nu_i}q(j,i) \varphi(x,j) - \sum_{j =1}^D q(i,j) \varphi(x,i) \right] dx=0.
\end{align*}
It follows that $\varphi$ is the solution of \eqref{eq:PDE-proof}.
\end{proof}

\subsection{The Ornstein-Uhlenbeck process}
\label{sec:OUprocess}
In this short section, we recall a classic property of multidimensional Ornstein-Uhlenbeck processes, which is useful to characterize the behavior of $(y_n)_{n\geq1}$ in a standard setting. Thus, we define $(\Y_t)_{t\geq0}$ as the strong solution of the following SDE, with values in $\R^D$:
\begin{equation}
\Y_t=\Y_0-\int_0^t \Y_{s^-}ds+\sqrt2\int_0^t (\Sigma^{(p,\Upsilon)})^{1/2} dW_s,
\label{eq:SdeOU}
\end{equation}
where $W$ is a standard $D$-dimensional Brownian motion and $(\Sigma^{(p,\Upsilon)})^{1/2}$ is the square root of the positive-definite symmetric matrix $\Sigma^{(p,\Upsilon)}$, i.e. $(\Sigma^{(p,\Upsilon)})^{1/2}((\Sigma^{(p,\Upsilon)})^{1/2})^\top=\Sigma^{(p,\Upsilon)}$. The process $\Y$ is a classic Ornstein-Uhlenbeck process with infinitesimal generator $\mathcal L_\text{O}$ defined in \eqref{eq:GeneratorOU}. Such processes have already been thoroughly studied, so we present only the following proposition, which quantifies the speed of convergence of $\Y$ to its equilibrium.
\begin{proposition}[Ergodicity of the Ornstein-Uhlenbeck process]\label{proposition:ErgodicityOU}
The Markov process $(\Y_t)_{t\geq0}$ generated by $\mathcal L_\text{O}$ in \eqref{eq:GeneratorOU}, with values in $\R$, admits a unique stationary distribution $\mathscr N\left(0,\Sigma^{(p,\Upsilon)}\right).$

Moreover,
\[W\left(\Y_t,\mathscr N\left(0,\Sigma^{(p,\Upsilon)}\right)\right)=W(\Y_0,\pi)\e^{-t}.\]
\end{proposition}

\begin{proof}[Proof of Proposition~\ref{proposition:ErgodicityOU}]
First, since
\[\mathscr N\left(0,\Sigma^{(p,\Upsilon)}\right)(dx)=C\exp\left(\frac{x^\top(\Sigma^{(p,\Upsilon)})^{-1}x}2\right)dx,\]
a straightforward integration by parts shows that, for any $f\in\mathscr C^2_c$, $\mathscr N\left(0,\Sigma^{(p,\Upsilon)}\right)(\mathcal L_\text{O}f)=0$
so that $\mathscr N\left(0,\Sigma^{(p,\Upsilon)}\right)$ is an invariant measure for the Ornstein-Uhlenbeck process $\Y$.

It is well-known and easy to check that $(\Y_t)_{t\geq0}$ writes
\[\Y_t=\Y_0\e^{-t}+\sqrt2(\Sigma^{(p,\Upsilon)})^{1/2}\int_0^t\e^{-(t-s)}dW_s,\]
where $W$ is a standard Brownian motion. Consequently, if we consider $\tilde\Y$ another Ornstein-Uhlenbeck process generated by $\mathcal L_\text{O}$ and driven by the (same) Brownian motion $W$,
\begin{equation}
\E\left[|\Y_t-\tilde\Y_t |\right]=\E\left[|\Y_0-\tilde\Y_0|\right]\e^{-t}.
\label{eq:proofErgodicityOU1}
\end{equation}
Taking the infimum over all the couplings gives a contraction in Wasserstein distance. Now, if $\mathscr L(\tilde\Y_0)=\mathscr N\left(0,\Sigma^{(p,\Upsilon)}\right)$ and $(\Y_0,\tilde\Y_0)$ is the optimal coupling between $\mathscr L(\Y_0)$ and $\mathscr N\left(0,\Sigma^{(p,\Upsilon)}\right)$ with respect to $W$, then \eqref{eq:proofErgodicityOU1} writes
\[W\left(\Y_t,\mathscr N\left(0,\Sigma^{(p,\Upsilon)}\right)\right)=W\left(\Y_0,\mathscr N\left(0,\Sigma^{(p,\Upsilon)}\right)\right)\e^{-t},\]
which entails the uniqueness of the invariant probability distribution as well as the exponential ergodicity of the process.
\end{proof}

\subsection{Acceleration of the jumps}
\label{sec:EZZtoOU}
The current section links the Sections~\ref{sec:EZZprocess} and \ref{sec:OUprocess} in the following sense:
\begin{center}
\begin{tikzpicture}[>=triangle 90,rounded corners]
	\node[blockR] (a) {Markov chain\\$(i_n)_{n\geq1}$};
	\node[blockR, above right=0cm and 4cm of a] (b) {Exponential zig-zag process\\$(\X_t,\I_t)_{t\geq0}$};
	\node[blockR, below right=0cm and 4cm of a]   (c){Ornstein-Uhlenbeck process\\$(\Y_t)_{t\geq0}$};

	\draw[->] (a.north) |- node[midway,above right]{Slow freezing} (b.west);
	\draw[->] (a.south) |- node[midway,below right]{Fast freezing} (c.west);
	\draw[dotted,->] (b.south) -- node[midway,right]{Acceleration of the jumps} (c.north);
\end{tikzpicture}
\end{center}

Indeed, we prove in Theorem~\ref{thm:CVEzzOu} the convergence of the (rescaled) exponential zig-zag process to a diffusive process as the jump rates go to infinity. Such results are fairly standard and are already known in the cases of (linear) zig-zag processes (see \cite{FGM12,BD16}) or of particle transport processes (see \cite{CK06}). Heuristically, since there are more frequent jumps, the process tends to concentrate around its mean $\nu$, and the effect of the discrete component fades away. This phenomenon can be seen on Figure~\ref{fig:EZZ}. We shall end this section with Corollary~\ref{coro:PitoNor}, which provides the convergence of the stationary distribution of the exponential zig-zag process toward a Gaussian distribution.

To this end, let $(a_n)_{n\geq1}$ be a sequence of positive numbers such that $a_n\to+\infty$ as $n\to+\infty$ and, for any integer $n$, let $(\X_t^{(n)},\I_t^{(n)})_{t\geq0}$ be a Markov process with values in $E$ generated by
\[\mathcal L^{(n)}f(x,i)=(e_i-x)\cdot\nabla_x f(x,i)+a_n\sum_{j\neq i}q(i,j)[f(x,j)-f(x,i)].\]
We define $\Y^{(n)}_t=\sqrt{a_n}(\X^{(n)}-\nu)$ and denote by $\Y^{(n)}_t(k)$ and $\X^{(n)}_t(k)$ the respective $k^\text{th}$ component of $\Y^{(n)}_t$ and $\X^{(n)}_t$. 

\begin{theorem}[Convergence of the processes]
\label{thm:CVEzzOu}
If $(\Y^{(n)}_0)_{n\geq1}$ converges in distribution to a probability distribution $\mu$, then the sequence of processes $(\Y^{(n)})_{n\geq1}$ converges in distribution to the diffusive Markov process generated by
\[\mathcal L_\text{O}f(y)=-y\cdot\nabla f(y)+\nabla f(y)^\top\Sigma^{(0,1)}\nabla f(y)\]
with initial condition $\mu$.
\end{theorem}

\begin{proof}[Proof of Theorem~\ref{thm:CVEzzOu}]
We shall use a diffusion approximation and follow the proof of \cite[Proposition~1.1]{FGM12}. For now, we drop the superscript $(n)$, and let, for any $1\leq k,l\leq D$,
\[\varphi_k(x,i)=\sqrt{a_n}(x_k-\nu_k)+\frac1{\sqrt{a_n}}h_{k,i},\quad\psi_{k,l}(x,i)=\varphi_k(x,i)\varphi_l(x,i).\]
Then,
\begin{align*}
\mathcal L\varphi_k(x,i)&=\sqrt{a_n}(\nu_k-x_k),\\
\mathcal L\psi_{k,l}(x,i)&=\sqrt{a_n}\left((\indic_{i=k}-x_k)\varphi_l(x,i)+(\indic_{i=l}-x_l)\varphi_k(x,i)\right)\\
&\quad+a_n\left((x_k-\nu_k)(\nu_l-\indic_{i=l})+(x_l-\nu_l)(\nu_k-\indic_{i=k})\right)+\sum_{j\neq i}q(i,j)\left(h_{k,j}h_{l,j}-h_{k,i}h_{l,i}\right).
\end{align*}
Then, by Dynkin's formula, for fixed $n$, the processes $(M_t(k))_{t\geq0}$ and $(N_t(k,l))_{t\geq0}$ are local martingales with respect to the filtration generated by $(\X^{(n)},\I^{(n)})$, where
\begin{align*}
M_t(k)&=\Y_t(k)-\sqrt{a_n}\int_0^t(\nu_k-\X_s(k))ds+\frac1{\sqrt{a_n}}h_{k,\I_t},\\
N_t(k,l)&=\Y_t(k)\Y_t(l)+\Y_t(k)h_{l,\I_t}+\Y_t(l)h_{k,\I_t}+\frac1{a_n}h_{k, \I_t}h_{l,\I_t}\\
&\quad-\int_0^t\big[-2\Y_s(l)\Y_s(k)+h_{k,\I_s}(\indic_{\{\I_s=l\}}-\X_s(l))+h_{l,\I_s}(\indic_{\{\I_s=k\}}-\X_s(k))\\
&\quad+\sum_{j\neq \I_s}q(\I_s,j)\left(h_{k,j}h_{l,j}-h_{k,\I_s}h_{l,\I_s}\right)\big]ds.
\end{align*}
Remark that, for any $1\leq i\leq D$, if $\sigma_{k,l}(i)=\sum_{j=1}^D q(i,j) (h_{l,j} - h_{l,i})(h_{k,j} - h_{k,i)} )$,
\begin{align*}
\sum_{j=1}^Dq(i,j)\left(h_{k,j}h_{l,j}-h_{k,i}h_{l,i}\right)&=\sum_{j=1}^Dq(i,j)\left(h_{l,j}-h_{l,i}\right)\left(h_{k,j}-h_{k,i}\right)+h_{l,i}\left(\nu_k-\indic_{i=k}\right)+h_{k,i}\left(\nu_l-\indic_{i=l}\right)\\
&=\sigma_{k,l}(i)+h_{l,i}\left(\nu_k-\indic_{i=k}\right)+h_{k,i}\left(\nu_l-\indic_{i=l}\right).\\
\end{align*}
Then, denoting by $\Z_s(k)=\int_0^t\Y_s(k)ds$,
\begin{align*}
N_t(k,l)&=\Y_t(k)\Y_t(l)+2\int_0^t\Y_s(k)\Y_s(l)ds-\int_0^t\sigma_{k,l}(\I_s)ds+\frac1{a_n}h_{k,\I_t}h_{l,\I_t}\\
&\quad+\frac1{\sqrt{a_n}}h_{k,\I_t}\left(\Y_t(l)+\Z_t(l)\right)+\frac1{\sqrt{a_n}}h_{l,\I_t}\left(\Y_t(k)+\Z_t(k)\right),
\end{align*}
and
\begin{align*}
M_t(k)M_t(l)&=N_t(k,l)+\Y_t(k)\Z_t(l)+\Y_t(l)\Z_t(k)+\Z_t(k)\Z_t(l)-2\int_0^t\Y_s(k)\Y_s(l)ds+\int_0^t\sigma_{k,l}(\I_s)ds\\
&\quad+\frac1{\sqrt{a_n}}\left(h_{k,\I_t}\Z_t(l)+h_{l,\I_t}\Z_t(k)-\int_0^th_{k,\I_s}\Y_s(l)ds-\int_0^th_{l,\I_s}\Y_s(k)ds\right).
\end{align*}
By integration by parts, 
\begin{align*}
\Y_t(k)\Z_t(l)&=\int_0^t\Z_s(l)dM_s(k)-\int_0^t\Z_s(l)\Y_s(k)ds+\int_0^t\Y_s(k)\Y_s(l)ds\\
&\quad+\frac1{\sqrt{a_n}}\left(\int_0^th_{k,\I_s}\Y_s(l)ds-h_{k,\I_t}\Z_s(l)\right),
\end{align*}
hence
\[M_t(k)M_t(l)=N_t(k,l)+\int_0^t\Z_s(k)dM_s(l)+\int_0^t\Z_s(l)dM_s(k)+\int_0^t\sigma_{k,l}(\I_s)ds.\]
Finally, for any $1\leq k,l\leq D$, the processes $M^{(n)}(k)-B^{(n)}(k)$ and $M^{(n)}(k)M^{(n)}(l)-A^{(n)}(k,l)$ are local martingales, with
\[A^{(n)}_t(k,l)=\int_0^t\sigma_{k,l}(\I^{(n)}_s)ds,\quad B^{(n)}_t(k)=-\int_0^t\Y^{(n)}_sds+\frac1{\sqrt{a_n}}h_{k,\I^{(n)}_t}.\]
Note that $\I^{(n)}$ is a Markov process on its own, generated by 
\[\mathcal L_\text{I}^{(n)}f(i)=a_n\sum_{j\neq i}q(i,j)[f(j)-f(i)].\]
In other words, for any $t>0$, we can write $\I^{(n)}_t=\I_{a_nt}$ a.s., for some pure-jump Markov process $(\I_t)_{t\geq0}$ generated by
\[\mathcal L_\text{I}f(i)=\sum_{j\neq i}q(i,j)[f(j)-f(i)].\]
Using the ergodicity of $(\I_t)_{t\geq0}$ together with $\lim_{n\to+\infty}a_n=+\infty$, we have
\[\lim_{n\to+\infty}A^{(n)}_t(k,l)=\lim_{n\to+\infty}\int_0^t\sigma_{k,l}(\I_{a_ns})ds=\lim_{n\to+\infty}\frac1{a_n}\int_0^{a_nt}\sigma_{k,l}(\I_u)du=t\sum_{i=1}^D\nu_i\sigma_{k,l}(i)=t\nu(\sigma_{k,l}).\]
Thus, the processes $\Y^{(n)}(k),B^{(n)}(k),A^{(n)}(k,l)$ satisfy the assumptions of \cite[Chapter~7, Theorem~4.1]{EK86}, which entails Theorem~\ref{thm:CVEzzOu}.
\end{proof}

\begin{remark}[Heuristics for a direct Taylor expansion of the generator]
As for many limit theorems for Markov processes, one would like to predict the convergence of the exponential zig-zag process to the Ornstein-Uhlenbeck diffusion from a Taylor expansion of the generator. Let us describe here a quick heuristic argument based on \cite{CK06}, which justifies the particular choice of functions $\varphi_k$ in the proof of Theorem~\ref{thm:CVEzzOu}. For the sake of simplicitylet us work in the setting of Section~\ref{sec:Turnover}, that is the generator of $(\X_t,\I_t)_{t\geq 0}$ is of the form
\begin{equation*}
\mathcal L_\text{Z}f(x,i)=g_i(x)\partial_xf(x,i)+a\theta_{3-i}[f(x,3-i)-f(x,i)]
\end{equation*}
where $g_i:x\mapsto  (\indic_{\{i=1\}}-x) $. 
For some smooth function $f:\R^D \to \R$, we have $\mathcal L_{\text{Z}} f(x,i)= g_i(x)\cdot\nabla_x f(x)$ which cannot be rescaled to converge to some diffusive operator. We need an approximation $f_a$ of $f$ in a sense that $\lim_{a\to+\infty} f_a =f$ and $\mathcal{L}_\text{Z} f_a$ has the form of a second order operator. Then, let
\[f_a:(x,i) \mapsto f(x) + a^{-1} h(x,i) \cdot \nabla_x f(x)\]
where $h(x,i)$ is the solution of the multidimensional Poisson equation associated to the transitions of the flows
\[\sum_{j\neq i}q(i,j)[h(x,j)-h(x,i)] = \theta_{3-i}[h(x,3-i)-h(x,i)] = \sum_j \nu_j g_j(x) - g_i(x) = \nu_1 - \indic_{\{i=1\}}.\]
Then,
\[\mathcal L_\text{Z} f_a(x,i)=\frac{1}{a} g_i(x)\cdot\nabla_x (h \cdot \nabla_x f)(x,i)+ \nabla_xf(x) \sum_j \nu_j g_j(x). \]
Here, $\sum_j \nu_j g_j(x) - g_i(x)=\nu_1 - \indic_{\{i=1\}}$ does not depend on $x$, neither does the function $h$, which is thus defined by \eqref{eq:DefH}. Furthermore, $h(x,i)=(\theta_1+\theta_2)^{-1}\indic_{i=1}$. Moreover $\lim_{a\to+\infty}g_i(\nu + y/\sqrt{a}) = e_i - \nu$, so $\lim_{a\to+\infty}\mathcal L_\text{Z} f_a(x,i)=\mathcal L_\text{O} f(x)$ up to renormalization.
\end{remark}

From Proposition~\ref{proposition:ErgodicityEZZ}, for any fixed $n\geq1$, the process $(\X_t^{(n)},\I_t^{(n)})_{t\geq0}$ admits and converges to a unique invariant distribution $\pi^{(n)}$, characterized in \eqref{eq:DecompositionPi} as
\[\pi^{(n)}=\sum_{i=1}^D\nu_i\pi_i^{(n)}\otimes\delta_i,\quad\pi_i^{(n)}(dx)=\varphi^{(n)}(x,i)dx.\]
Let $\bar\pi^{(n)}$ be the first margin of the invariant measure of the Markov process $(\Y_t^{(n)},\I_t^{(n)})_{t\geq0}$, i.e. the probability distribution over $\R^D$ defined by
\[\bar\pi^{(n)}(dy)=\sum_{i=1}^D\frac{\nu_i}{\sqrt{a_n}}\varphi^{(n)}\left(\frac y{\sqrt{a_n}}+\nu,i\right)dy.\]

\begin{corollary}[Convergence of the stationary distributions]
\label{coro:PitoNor}
The sequence of probability measures $(\bar\pi^{(n)})_{n\geq1}$ converges to $\mathscr N\left(0,\Sigma^{(0,1)}\right)$.
\end{corollary}

\begin{proof}[Proof of Corollary~\ref{coro:PitoNor}]
Let  $n\geq1,t\geq0$ and
$$
\mathscr{F} = \left\{ f\in \mathscr C_c^2(\R^D) \ : \ \Vert f \Vert_\infty \leq  1 , \left|f(x) - f(y)\right| \leq |x-y|  \right\}.
$$
Up to a constant, $d_\mathscr{F}$ is the Fortet-Mourier distance and metrizes the weak convergence. Fix $t\geq 0$ and let $\X_0^{(n)}=\nu$ and $\mathscr L(\I^{(n)}_0)= \nu^\top$. From Theorem~\ref{thm:CVEzzOu},
$$
\lim_{n \to+\infty} d_{\mathscr{F}} \left(\Y_t^{(n)},\Y_t\right)=0,
$$
where $\Y$ is an Ornstein-Uhlenbeck process with generator $\mathcal L_\text{O}$ and initial condition $0$. Using the definition of $d_{\mathscr{F}}$ and Proposition~\ref{proposition:ErgodicityEZZ},
$$
d_{\mathscr{F}}\left(\Y_t^{(n)}, \bar{\pi}^{(n)}\right) \leq W\left((\Y^{(n)}_t,\I^{(n)}_t), \pi^{(n)}\right) \leq W\left(\delta_0 \otimes \nu, \pi^{(n)}\right) e^{-t}=W\left(\delta_0, \bar\pi^{(n)}\right) e^{-t}.
$$
Let us check that the term $W\left(\delta_0, \bar\pi^{(n)}\right)$ is uniformly bounded. To that end, let 
\[f^{(n)}(x,i) = x^2_k +  \frac{2}{a_n}  h_{k,i} x_k + 2 \nu_k x_k,\]
so that
\[\mathcal{L}^{(n)} f^{(n)}(x,i)=  - 2 x_k^2 +  \frac{2}{a_n} \left(\indic_{\{i=k\}} -x_k\right) h_{k,i} + 2 \nu_k \indic_{\{i=k\}}.\]
Since $\pi^{(n)}\left(\mathcal L^{(n)}f^{(n)}\right)=0$,
\[\int_E x_k^2 \pi^{(n)}(dx,di)- \nu_k^2 =\frac{1}{a_n} \left(h_{k,k} \nu_k - \int_E x_k h_{k,i} \pi^{(n)}(dx,di)\right).\]
Hence, with $C=\sum_{k=1}^D h_{k,k} \nu_k - \min_{i,j} h_{i,j}$, and since $\int_E x_k \pi^{(n)}(dx,di)=\nu_k$,
\begin{align*}
\int_E \|x-\nu\|^2_2 \pi^{(n)}(dx,di)
&= \sum_{k=1}^D \int_E (x_k- \nu_k)^2 \pi^{(n)}(dx,di)= \sum_{k=1}^D \int_E \left(x_k^2 - 2 \nu_k x_k + \nu_k^2 \right)\pi^{(n)}(dx,di) \\
&= \sum_{k=1}^D \int_E x_k^2 \pi^{(n)}(dx,di) -  \nu_k^2 = \frac1{a_n} \left( \sum_{k=1}^D h_{k,k} \nu_k - \int_E x_k h_{k,i} \pi^{(n)}(dx,di)\right) \\
&\leq \frac{1}{a_n} \left(\sum_{k=1}^D h_{k,k} \nu_k - \min_{i,j} h_{i,j}\right)\leq\frac C{a_n}.
\end{align*}
By H\"older's inequality,
\[W\left(\delta_0, \bar\pi^{(n)}\right)=\int_{\R}|y|\bar\pi^{(n)}(dy)=\int_E \sqrt{a_n}|x-\nu| \pi^{(n)}(dx,di)\leq\sqrt C.\]
Consequently to Proposition~\ref{proposition:ErgodicityOU},
\begin{align*}
d_{\mathscr{F}} \left(\bar{\pi}^{(n)}, \mathscr N\left(0,\Sigma^{(0,1)}\right)\right)
&\leq d_{\mathscr{F}} \left(\bar{\pi}^{(n)},\Y_t^{(n)} \right) + d_{\mathscr{F}} \left(\Y_t^{(n)},\Y_t \right)+  d_{\mathscr{F}} \left(\Y_t, \mathscr N\left(0,\Sigma^{(0,1)}\right) \right)\\
&\leq  2 \sqrt C e^{-t} + d_\mathscr{F} \left(\Y_t^{(n)},\Y_t\right).
\end{align*}
Then,
\[\limsup_{t\to+\infty}d_{\mathscr{F}} \left(\bar{\pi}^{(n)}, \mathscr N\left(0,\Sigma^{(0,1)}\right)\right)\leq 2 \sqrt C e^{-t},\]
which goes to 0 as $t\to+\infty$.
\end{proof}

\section{Complete graph}
\label{sec:Applications}
In this section, we consider a particular case of freezing Markov chain, where all the states are connected, and the jump rate to a state does not depend on the position of the chain. This example of Markov chain has already been studied in the literature, for instance in \cite{DS07}. Section~\ref{sec:Dirichlet} deals with the general $D$-dimensional case, for which most of the results of Section~\ref{sec:MarkovProcesses} can be written explicitly, notably the invariant measure of the exponential zig-zag process, which is a mixture of Dirichlet distributions (see Figure~\ref{fig:Dirichlet}). Section~\ref{sec:Turnover} studies more deeply the case $D=2$, where we can refine the speed of convergence provided in Proposition~\ref{proposition:ErgodicityEZZ}.

\subsection{General case}
\label{sec:Dirichlet}
Throughout this section, following \cite{DS07}, we assume that there exists a positive vector $\theta\in(0,+\infty)^D$ such that, for any $1\leq i,j\leq D$,
\begin{equation}
q(i,j)=\theta_j-|\theta|\indic_{i=j}, \quad|\theta|=\sum_{j=1}^D\theta_j,
\label{eq:qDirichlet}
\end{equation}
and we will recover \cite[Theorem~1.4]{DS07}. If $D=2$, let us highlight that an irreducible matrix $\Id+q$ automatically satisfies \eqref{eq:qDirichlet} (if $\Id+q$ is indecomposable then this is true as soon as $q(1,2)q(2,1)\neq0$).

\begin{figure}[htbp!]
\begin{center}
\hspace*{\fill}
\includegraphics[width=0.29\textwidth]{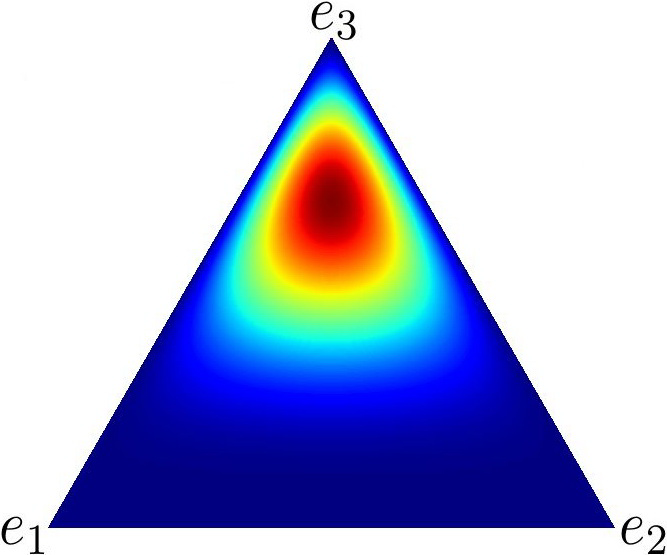}\hfill
\includegraphics[width=0.29\textwidth]{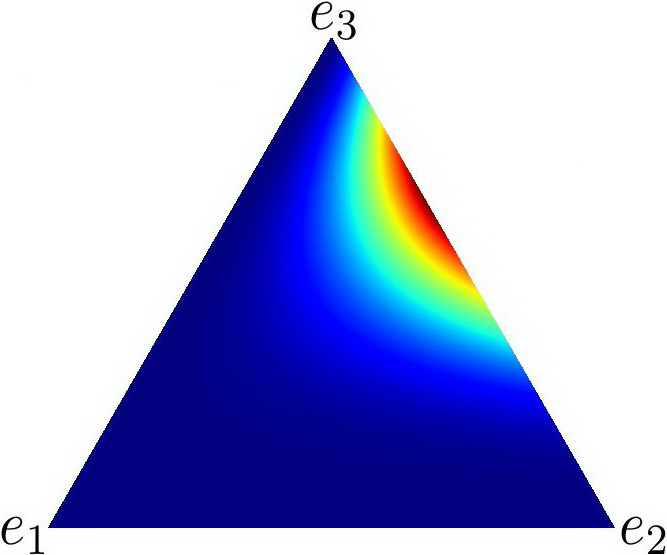}\hfill
\includegraphics[width=0.29\textwidth]{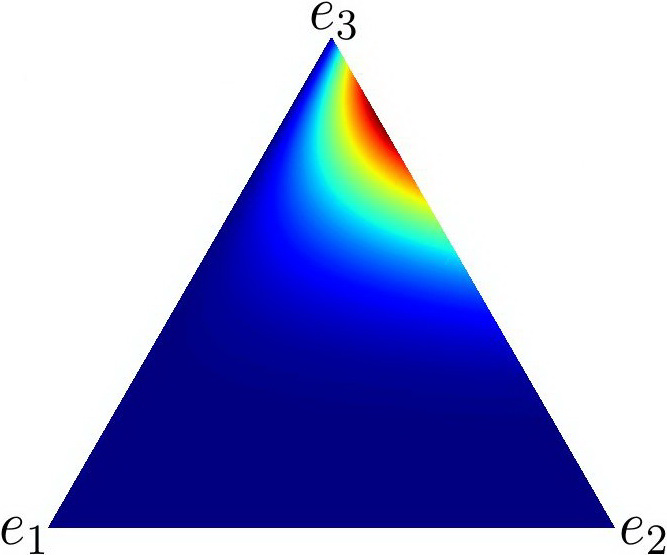}
\hspace*{\fill}
\end{center}
\caption{Probability density functions of $\pi_1=\Dir(2,2,5),\pi_2=\Dir(1,3,5),\pi_3=\Dir(1,2,6)$, for $\theta_1=1,\theta_2=2,\theta_3=5$.}
\label{fig:Dirichlet}\end{figure}

\begin{proposition}[Limit distribution for the complete graph in the non-standard setting]
\label{prop:DirichletEZZ}
Under Assumptions~\ref{assumption:FreezingSpeed} and \ref{assumption:NonStandard}, and if $q$ satisfies \eqref{eq:qDirichlet}, then $\nu_i=\theta_i|\theta|^{-1}$ and
\[\lim_{n\to+\infty}(x_n,i_n)=\sum_{i=1}^D\nu_i\Dir(a\theta+e_i)\otimes\delta_i\text{ in distribution}.\]
In particular,
\[\lim_{n\to+\infty}x_n=\Dir(a\theta)\text{ in distribution,}\quad\lim_{n\to+\infty}i_n=\nu^\top\text{ in distribution.}\]
\end{proposition}

\begin{proof}[Proof of Proposition~\ref{prop:DirichletEZZ}]
If $q$ satisfies \eqref{eq:qDirichlet}, it is straightforward that its invariant distribution $\nu^\top$ is given by $\nu_i = \theta_i|\theta|^{-1}$ for any $1\leq i \leq D$. The convergence of $(i_n)_{n\geq1}$ to $\nu^\top$ and of $(x_n,i_n)_{n\geq1}$ to some distribution $\pi$ are direct corollaries of Theorems~\ref{thm:CVMarkovChain} and \ref{thm:NonStandard}. Moreover, Proposition~\ref{prop:Pi_PDE} holds, hence $\pi$ satisfies \eqref{eq:DecompositionPi} and it is clear that
\[\varphi(x,i)=\frac{\Gamma(|\theta|+1)}{\Gamma(\theta_i+1)\prod_{j\neq i}\Gamma(\theta_j)}x_i^{\theta_i}\prod_{j\neq i}x_j^{a\theta_j-1}=\nu_i\frac{\Gamma(|\theta|)}{\prod_{j=1}^D\Gamma(\theta_j)}x_i^{\theta_i}\prod_{j\neq i}x_j^{a\theta_j-1}\]
is the unique (up to a multiplicative constant) solution of \eqref{eq:PDE-proof}, which entails that
\[\pi=\sum_{i=1}^D\nu_i\Dir(a\theta+e_i)\otimes\delta_i.\]
Finally, if $\mathscr L(X,I)=\pi$, it is clear that $\mathscr L(I)=\nu^\top$ and that
\[\mathscr L(X)(dx)=\sum_{i=1}^D\nu_i\varphi(x,i)dx=\frac{\Gamma(|\theta|)}{\prod_{j=1}^D\Gamma(\theta_j)}\prod_{j\neq i}x_j^{a\theta_j-1}dx=\Dir(\theta)(dx).\]
\end{proof}

In the framework of \eqref{eq:qDirichlet}, it is also possible to obtain explicitly the solution of the Poisson equation related to $q$ as well as the covariance matrix of the limit distribution in the standard setting. This is the purpose of the following result, whose proof is straightforward using Theorem~\ref{thm:Standard} together with the expressions \eqref{eq:DefSigma} and \eqref{eq:DefH}.

\begin{proposition}[Limit distribution for the complete graph in the standard setting]
\label{prop:DirichletOU}
Under Assumptions~\ref{assumption:FreezingSpeed} and \ref{assumption:Standard}, and if $q$ satisfies \eqref{eq:qDirichlet}, then $\nu=|\theta|^{-1}\theta$ and $h_i=|\theta|^{-1}e_i$ and
\[\lim_{n\to+\infty}y_n=\mathscr N\left(0,\Sigma^{(p,\Upsilon)}\right)\text{ in distribution},\quad \text{with }\Sigma^{(p,\Upsilon)}_{k,l}=\left\{\begin{array}{ll}
-\frac{2-p}{1+\Upsilon}\nu_k\nu_l&\text{ if }k\neq l\\
\frac{2-p}{1+\Upsilon}\nu_k(1-\nu_k)&\text{ if }k= l
\end{array}\right..
\]
\end{proposition}

Finally, let us emphasize the fact that Corollary~\ref{coro:PitoNor} provides an interesting convergence of rescaled Dirichlet distributions, when considered in the particular case of the complete graph.

\begin{corollary}[Convergence of the rescaled Dirichlet distribution to a Gaussian law]
\label{coro:DirichletPitoNor}
For any vector $\theta\in(0,+\infty)^D$, if $(X_n)_{n\geq1}$ is a sequence of independent random variables such that $\mathscr L(X_n)=\mathscr D(a_n\theta)$, then
\[\lim_{n\to+\infty}\sqrt{a_n}\left(X_n-\nu\right) =\mathscr N\left(0,\diag(\nu)-\nu\nu^\top\right)\text{ in distribution}.\]
\end{corollary}

\subsection{The turnover algorithm}
\label{sec:Turnover}
In this subsection, we consider the turnover algorithm introduced in \cite{EV16}. This algorithm studies empirical frequency of \emph{heads} when a coin is turned over with a certain probability, instead of being tossed as usual. The authors provide various convergences in distribution for this proportion, depending on the asymptotic behavior of the turnover probability, which corresponds to $(p_n)_{n\geq1}$ in the present paper. However, this turnover algorithm can be seen as a particular case of freezing Markov chain, and can then be written as the stochastic algorithm defined in \eqref{eq:DefXAlgoSto}, in the special case $D=2$. Since $x_n(1)=1-x_n(2)$, there is only one relevant variable in this section, which belongs to $[0,1]$:
\begin{equation}
x_n=x_n(1)=\gamma_n\sum_{k=1}^n\indic_{\{i_k=1\}}.
\label{eq:TurnoverAlgo}
\end{equation}

Note that we are in the framework of Section~\ref{sec:Dirichlet}, with $\theta_1=q(2,1)$ and $\theta_2=q(1,2)$, and that Propositions~\ref{prop:DirichletEZZ} and \ref{prop:DirichletOU} hold. In particular, we have $\nu_i=\theta_i(\theta_1+ \theta_2)^{-1}$. Then, for any $y\in\R$ and $(x,i)\in[0,1]\times\{1,2\}$, the infinitesimal generators defined in \eqref{eq:GeneratorOU} and \eqref{eq:GeneratorEZZ} write
\begin{equation}
\label{eq:GeneratorOU2}
\mathcal L_\text{O}f(y)=-yf'(y)+\frac {2-p}{1+\Upsilon}\nu_1(1-\nu_1)f''(y)
\end{equation}
and
\begin{equation}
\label{eq:GeneratorEZZ2}
\mathcal L_\text{Z}f(x,i)=(\indic_{\{i=1\}}-x)\partial_xf(x,i)+a\theta_{3-i}[f(x,3-i)-f(x,i)].
\end{equation}

\begin{remark}[Comparison with \cite{EV16}]
In the present paper, we recover \cite[Theorems~1 and 2]{EV16} as direct consequences of Theorems~\ref{thm:NonStandard} and \ref{thm:Standard}. The aforementioned results are extended by allowing $q(1,2)\neq q(2,1)$, but mostly by obtaining results for general sequences $(p_n)_{n\geq 1}$ while \cite{EV16} deals only with $p_n=an^{-\theta}$ for positive constants $a$ and $\theta$. It should be noted that, in order to perfectly mimic the algorithm of the aforementioned article, one should consider the chain $x^\star_n=\gamma_n\sum_{k=1}^n\left(\indic_{\{i_k=1\}}-\indic_{\{i_k=2\}}\right)$, which evolves in $[-1,1]$. The behavior of this sequence being completely similar to the one we are studying, we chose to work with \eqref{eq:TurnoverAlgo} for the sake of consistence.

However, the reader should notice that the invariant measure of the process generated by \eqref{eq:GeneratorOU2} is a Gaussian distribution with variance $\Sigma^{(p,\Upsilon)}_{1,1}$. In the particular case where $p=0$ and $\theta_1=\theta_2$, this variance writes
\[\Sigma^{(0,\Upsilon)}_{1,1}=\frac1{2(1+\Upsilon)},\]
which is, at first glance, different from the variance provided in \cite{EV16}, which is (under our notation)
\[\sigma^2=\frac1{a^2(1+\Upsilon)}.\]
The factor $a^2$ comes from the fact that \cite{EV16} studies the behavior of $a^{-1}y_n$. The factor 2 comes from the choice of normalization mentioned earlier, since $x_n\in[0,1]$ and $x_n^\star\in[-1,1]$.
\end{remark}

Whenever $D=2$, it is easier to visualize the dynamics of $(\X,\I)$ (see Figure~\ref{fig:EZZTurnover}), and we can improve the results of Proposition~\ref{proposition:ErgodicityEZZ} concerning the speed of convergence of the exponential zig-zag process to its stationary measure $\pi$.

\begin{figure}[htbp!]
\centering
\begin{tikzpicture}[xscale=2,yscale=3.2]
\draw[->] (0,0) node[left]{0} -- (7.2,0) node[below]{$t$};
\draw[->] (0,0) -- (0,1.2);
\draw (0,1) node[left]{1} -- (7.2,1);
\draw[blue] (0,0.5) node[left]{$\X_0$};
\draw[domain=0:1,smooth,blue] plot (\x,{0.5*exp(-\x)});
\draw [decorate,decoration={brace,amplitude=10pt,mirror},yshift=-0.2cm] (0,0) -- (1,0) node[midway,yshift=-0.5cm]{\footnotesize{$I_t=2$}};
\draw[dashed] (1,0) node[below]{$T_1$} -- (1,.1839);
\draw[domain=1:4,smooth,red] plot (\x,{1-(1-.1839)*exp(-(\x-1))});
\draw [decorate,decoration={brace,amplitude=10pt,mirror},yshift=-0.2cm] (1,0) -- (4,0) node[midway,yshift=-0.5cm]{\footnotesize{$I_t=1$}};
\draw[dashed] (4,0) node[below]{$T_2$} -- (4,.9594);
\draw[domain=4:5.5,smooth,blue] plot (\x,{.9594*exp(-(\x-4))});
\draw [decorate,decoration={brace,amplitude=10pt,mirror},yshift=-0.2cm] (4,0) -- (5.5,0) node[midway,yshift=-0.5cm]{\footnotesize{$I_t=2$}};
\draw[dashed] (5.5,0) node[below]{$T_3$} -- (5.5,.2141);
\draw[domain=5.5:7.2,smooth,red] plot (\x,{1-(1-.2141)*exp(-(\x-5.5))}) node[right]{$\X_t$};
\end{tikzpicture}
\caption{Typical path of the exponential zig-zag process when $D=2$.}
\label{fig:EZZTurnover}
\end{figure}
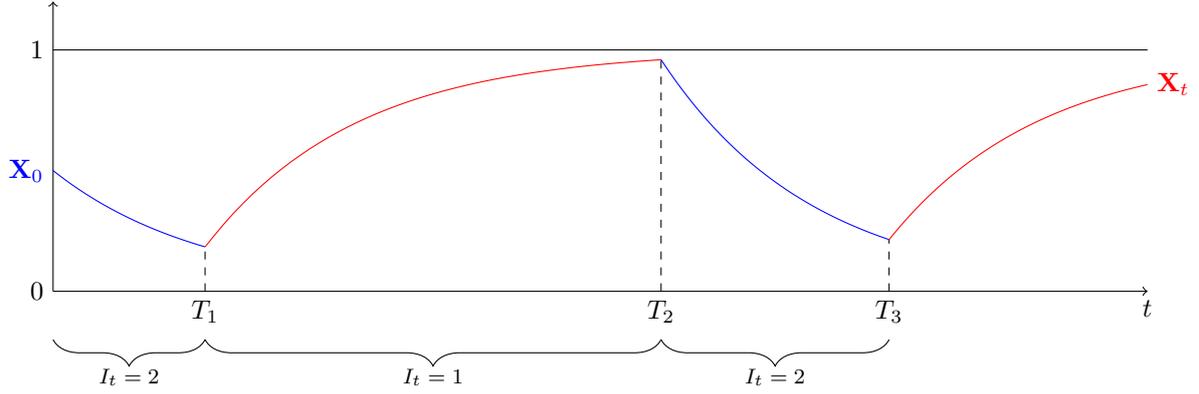

\begin{proposition}[Ergodicity when $D=2$]\label{proposition:ErgodicityEZZ2}
The Markov process $(\X_t,\I_t)_{t\geq0}$ generated by $\mathcal L_\text{Z}$ in \eqref{eq:GeneratorEZZ2}, with values in $[0,1]\times\{1,2\}$, admits a unique stationary distribution
\[\pi=\frac{\theta_1}{\theta_1+\theta_2}\beta(a\theta_1+1,a\theta_2)\otimes\delta_1+\frac{\theta_2}{\theta_1+\theta_2}\beta(a\theta_1,a\theta_2+1)\otimes\delta_2.\]

Moreover, let $v=a(\theta_1\vee\theta_2)$, then
\[W((\X_t,\I_t),\pi)\leq
\left\{\begin{array}{ll}
\left(2+\frac{2v}{|1-v|}\right)\e^{-(1\wedge v) t}&\text{if }v\neq1\\
(2+t)\e^{-t}&\text{if } v=1\\
W((\X_0,\I_0),\pi)\e^{-t}&\text{if }\mathscr L(\I_0)=\frac{\theta_1}{\theta_1+\theta_2}\delta_1+\frac{\theta_2}{\theta_1+\theta_2}\delta_2
\end{array}\right.
.\]
\end{proposition}

Since the inter-jump times of the exponential zig-zag process are spread-out, it is also possible to show convergence in total variation with a method similar to \cite[Proposition~2.5]{BMPPR15}. Note that, following Proposition~\ref{prop:DirichletEZZ}, the limit distribution of $(X_t)_{t\geq0}$ is the first margin of $\pi$, namely $\beta(a\theta_1,a\theta_2)$.

\begin{proof}[Proof of Proposition~\ref{proposition:ErgodicityEZZ2}]
Without loss of generality, let us assume that $\theta_1\geq\theta_2$, that is $v=a\theta_1$. Using Proposition~\ref{prop:DirichletEZZ}, it is clear that $\pi$ is the limit distribution of $(\X,\I)$. Let us turn to the quantification of the ergodicity of the process. Since the flow is exponentially contracting at rate 1, one can expect the Wasserstein distance of the spatial component $\X$ to decrease exponentially. The only issue is to bring $\I$ to its stationary measure first. So, consider the Markov coupling $\left((\X,\I),(\widetilde \X,\widetilde \I)\right)$ of $\mathcal L_\text{Z}$ on $E\times E$, which evolves independently if $\I\neq\tilde\I$, and else follows the same flow with common jumps. We set $T_0=0$ and denote by $T_n$ the epoch of its $n^\text{th}$ jump. If $\I_0\neq\widetilde \I_0$, the first jump is not common a.s., but in any case, since $D=2,$ $\I_{T_1}=\widetilde \I_{T_1}$ a.s. and $\mathscr L(T_1)=\mathscr E(v)$. Consequently,
\begin{align*}
\E\left[|(\X_t,\I_t)-(\tilde \X_t,\tilde \I_t)|\right]&=\E\left[|\X_t-\widetilde \X_t|\right]+\P(\I_t\neq\widetilde \I_t)\\
&\leq\int_0^t\E\left[\left.|\X_t-\widetilde \X_t|\right|T_1=s\right]v\e^{-vs}ds+\left(\E\left[\left.|\X_t-\widetilde \X_t|\right|T_1>t\right]+1\right)\P(T_1>t)\\
&\leq2\e^{-vt}+\int_0^t\E\left[|\X_s-\widetilde \X_s|\right]\e^{-(t-s)} v\e^{- vs}ds\\
&\leq2\e^{-v t}+v\e^{-t}\int_0^t\e^{(1-v)s}ds\\
&\leq \left[\left(2+\frac{v}{1-v}\right)\e^{-v t}-\frac{v}{1-v}\e^{-t}\right]\indic_{\{v\neq1\}}+(2+v t)\e^{-v t}\indic_{\{v=1\}}\\
&\leq \left(2+\frac{2v}{|1-v|}\right)\e^{-(1\wedge v) t}\indic_{\{v\neq1\}}+(2+t)\e^{-t}\indic_{\{\lambda=1\}}.
\end{align*}
Note that if $\mathscr L(\I_0)=\mathscr L(\tilde \I_0)$, let $\I_0=\tilde\I_0$, so that the coupling $\left((\X,\I),(\widetilde \X,\widetilde \I)\right)$ always has common jumps and
\[|\X_t-\widetilde \X_t|=|\X_0-\widetilde \X_0|\e^{-t}\text{ a.s.}\]
Letting $(\X_0,\widetilde \X_0)$ be the optimal Wasserstein coupling entails Wasserstein contraction. The results above hold for any initial conditions $(\tilde \X_0,\tilde \I_0)$. Then, let $\mathscr L(\tilde \X_0,\tilde \I_0)=\pi$ to achieve the proof; in particular, $\mathscr L(\tilde\I_0)=\nu^\top=\frac{\theta_1}{\theta_1+\theta_2}\delta_1+\frac{\theta_2}{\theta_1+\theta_2}\delta_2$.
\end{proof}

\section{Proofs}
\label{sec:proofs}
In this section, we provide the proofs of the main results of this paper that were stated throughout Section~\ref{sec:MainResults}.

\begin{proof}[Proof of Theorem~\ref{thm:CVMarkovChain}]
Under Assumption~\ref{assumption:FreezingSpeed}, let us first assume that $p>0$. The matrix $(\Id+q)$ is irreducible, and so is $(\Id+pq)$. Moreover, $\nu^\top$ is also the invariant measure of $pq$, and Perron-Frobenius Theorem entails that there exist $C>0$ and $\rho \in (0,1)$ such that for every $n\geq 1$ and $i\in\{1,\dots, D\}$,
$$
d_{\text{TV}}\left( \delta_i(\Id+pq)^n , \nu^\top \right) \leq C \rho^n.
$$
Now, let us prove that $(i_n)_{n\geq1}$ is an asymptotic pseudotrajectory of the dynamical system induced by $\Id+pq$. The limit set of such a system being contained in every global attractor (see \cite[Theorems~6.9 and 6.10]{Ben99}), we have
\begin{align}
d_\text{TV}\left(\delta_{i_n}(\Id+pq),i_{n+1}\right)&=d_\text{TV}\left(\delta_{i_n}(\Id+pq),\delta_{i_n}(\Id+q_n)\right)\notag\\
&\leq|p_n-p|+\sum_{j\neq i_n}|r_n(i_n,j)|\leq|p_n-p|+\sum_{i,j=1}^D|r_n(i,j)|,
\label{eq:proofCVMarkovChain1}
\end{align}
and the right-hand side of \eqref{eq:proofCVMarkovChain1} converges to 0, which ends the proof.

The case $p=0$ is a mere application of \cite[Proposition~3.13]{BBC17}.
\end{proof}

\subsection{Asymptotic pseudotrajectories in the non-standard setting}
\label{sec:NonStandard}

In this section, we prove Theorem~\ref{thm:NonStandard} using results from \cite{BBC17}, based on the theory of asymptotic pseudotrajectories for inhomogeneous-time Markov chains. Indeed, with the convention $\sum_{k=1}^0=0$, let
\begin{equation}
\tau_n=\sum_{k=1}^n\gamma_k,\quad m(t)=\sup\{k\geq0:\tau_k\leq t\},
\label{eq:DefTauM}
\end{equation}
and define the piecewise-constant processes
\begin{equation}
X_t=\sum_{n=1}^\infty x_n\indic_{\tau_n\leq t<\tau_{n+1}},\quad I_t=\sum_{n=1}^\infty i_n\indic_{\tau_n\leq t<\tau_{n+1}}.
\label{eq:interpolatedprocess}
\end{equation}

We shall show that, as $t\to+\infty$, the process $(X_t,I_t)_{t\geq0}$ converges in a way (see Figure~\ref{fig:interpolation}) to the exponential zig-zag process $(\X_t,\I_t)_{t\geq0}$ solution of \eqref{eq:SdeEZZ}, that we already studied in Section~\ref{sec:EZZprocess}. To that end, let $(P_t)_{t\geq0}$ be the Markov semigroup of $(\X,\I)$, $N_1=(2,\dots,2,0)$ and
\begin{equation}
\mathscr F=\left\{f\in\mathcal D(\mathcal L_\text{Z})\cap\mathscr C^{N_1}_b:\mathcal L_\text{Z}f\in\mathcal D(\mathcal L_\text{Z}),\|\mathcal L_\text{Z}f\|_\infty+\|\mathcal L_\text{Z}\mathcal L_\text{Z}f\|_\infty+\sum_{j=0}^{N_1}\|  f^{(j)} \|_\infty\leq1\right\}.
\label{eq:defF}
\end{equation}
Note that convergence with respect to $d_{\mathscr F}$ implies convergence in distribution (see \cite[Lemma~5.1]{BBC17}).

\begin{figure}[htbp!]
\begin{center}
\includegraphics[width=\textwidth]{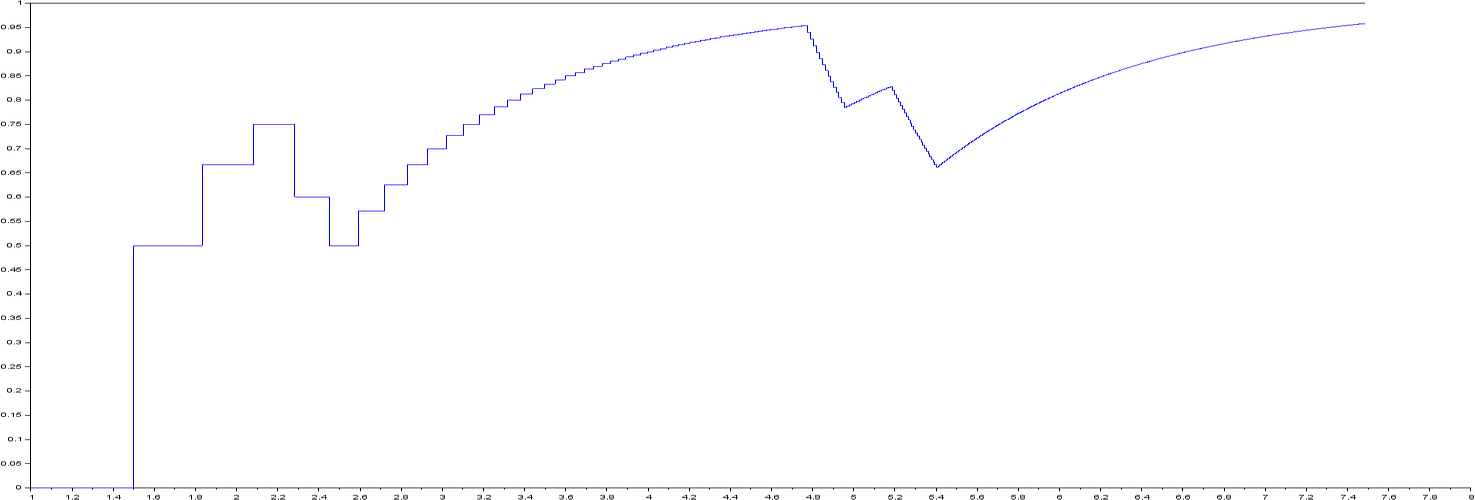}
\end{center}
\caption{Sample path of the process $(X_t)_{t\geq0}$ in the setting of Section~\ref{sec:Turnover} for $a=1$, $q(1,2)=\frac13$ and $q(2,1)=\frac23$.}
\label{fig:interpolation}\end{figure}

\begin{lemma}[Asymptotic pseudotrajectory for non-standard fluctuations]\label{lem:APTNonStandard}
Under the assumptions of Theorem~\ref{thm:NonStandard}, the sequence of probability distributions $(\mu_t)_{t\geq0}$ is an asymptotic pseudotrajectory of $(P_tf)_{t\geq0}$ with respect to $d_\mathscr F$.

Moreover, if there exist positive constants $A\geq1,\theta\leq1$ such that
\[\max_{j\neq i}(|r_n(i,j)|)\leq \frac A{n^{\theta}},\]
then, for any $v<a\rho(1+a\rho)^{-1}$, there exists a positive constant $C$ such that
\begin{equation}
d_{\mathscr F\cap\mathscr G}(\mathscr L(X_t,I_t),\pi)\leq C\e^{-vt}.
\label{eq:ProofThmNonStandard1}
\end{equation}

Moreover, the sequence of processes $\left((X_{n+t},I_{n+t})_{t\geq0}\right)_{n\geq1}$ converges in distribution, as $t\to+\infty$, toward $(\X^\pi_t,\I^\pi_t)_{t\geq0}$ in the Skorokhod space, where $(\X^\pi,\I^\pi)$ is a process generated by $\mathcal L_\text{Z}$ with initial condition $\pi$.
\end{lemma}

The proof of Lemma~\ref{lem:APTNonStandard} consists in checking \cite[Assumptions~2.1, 2.2, 2.3 and 2.7.ii)]{BBC17} and relies on three ingredients:
\begin{itemize}
	\item Convergence of a kind of discrete infinitesimal generator $\mathcal L_n$ , which characterizes the dynamics of $(X,I)$, to $\mathcal L_\text{Z}$ defined in \eqref{eq:GeneratorEZZ}.
	\item Smoothness of the limit semigroup $(P_t)_{t\geq0}$ and control of its derivatives with respect to the initial condition of the process.
	\item Uniform boundedness of the moments of $(x_n,i_n)_{n\geq1}$ up to some order, which is trivially satisfied here since $E$ is compact.
\end{itemize}

\begin{proof}[Proof of Lemma~\ref{lem:APTNonStandard}]
In what follows, the notation $O$ (as $n\to+\infty$) is uniform over $x,i,f$. We define $\mathcal L_nf(x,i)=\gamma_{n+1}^{-1}\E[f(x_{n+1},i_{n+1}|x_n=x,i_n=i]$, and we study the convergence of $\mathcal L_n$ to $\mathcal L_\text{Z}$ in the sense of \cite{BBC17}. Let $(x,i)\in E$ and $\chi_{N_1}(x,i)=\prod_{k=1}^Dx_k^2$. We recall that $q_n(i,i)=1+O(p_n)$ as $n\to+\infty$. With
\[\epsilon_n=\frac1n\vee\max_{j\neq i}(|r_n(i,j)|),\]
we have
\begin{align*}
\mathcal L_nf(x,i)&=
\sum_{j\neq i} \frac{q_{n+1}(i,j)}{\gamma_{n+1}}\left[f\left(\frac{\gamma_{n+1}}{\gamma_{n}}x+\gamma_{n+1}e_j,j\right)-f(x,i)\right]\\
&\quad+ \frac{1-\sum_{j\neq i}q_{n+1}(i,j)}{\gamma_{n+1}}\left[f\left(\frac{\gamma_{n+1}}{\gamma_{n}}x+\gamma_{n+1}e_i,i\right)-f(x,i)\right]\\
&=
\sum_{j\neq i} \frac{p_{n+1}}{\gamma_{n+1} }(q(i,j) + r_{n+1}(i,j)) \left[f(x,j)-f(x,i)+\chi_{(1,\dots,1,0)}(x,i)\|f^{(1,\dots,1,0)}\|_\infty O\left(\gamma_{n+1}\right)\right]\\
&\quad+  \frac{ 1+ O(p_{n+1})}{\gamma_{n+1}}\left[\left(\left(\frac{\gamma_{n+1}}{\gamma_n}-1\right)+\gamma_{n+1}e_i\right)\cdot\nabla_xf(x)+\chi_{N_1}(x,i)\|f^{(N_1)}\|_\infty O\left(\gamma_{n+1}^2\right)\right]\\
&=\mathcal L_{\text Z}f(x,i)+\chi_{N_1}(x,i)\|f^{(N_1)}\|_\infty O\left(\epsilon_{n+1}\right).
\end{align*}
We turn to the study of the regularity of the limit semigroup, following \cite{Kun84}. Let $t>0$ and note that $\|P_tf\|_\infty\leq\|f\|_\infty$. Moreover, the process $(\X,\I)$ is solution of the following SDE (we emphasize below the dependence on the initial condition):
\begin{equation}
(\X^{x,i}_t,\I^{x,i}_t)=(x,i)+\int_0^t \left(A(\X^{x,i}_{s^-},\I^{x,i}_{s^-})+e_{\I^{x,i}_{s^-}}\right)ds+\sum_{j=1}^D\int_0^t B_{\I^{x,i}_{s^-},j}(\X^{x,i}_{s^-},\I^{x,i}_{s^-})N_{\I^{x,i}_{s^-},j}(ds),
\label{eq:ProofThmNonStandard}
\end{equation}
where $N_{i,j}$ is a Poisson process of intensity $aq(i,j)\indic_{\{i\neq j\}}$ and the matrices $A$ and $B_{i,j}$ are defined in \eqref{eq:SdeEzz2}. Then, if we denote by $\eta_t^{x,i,k,h}=h^{-1}\left[(\X^{x+he_k,i}_t,\I^{x+he_k,i}_t)-(\X^{x,i}_t,\I^{x,i}_t)\right]$, we recover from \eqref{eq:ProofThmNonStandard} that the process $\eta^{x,i,k,h}$ satisfies the ODE
\[\eta_t^{x,i,k,h}=(e_k,0)+\int_0^t A\eta_{s^-}^{x,i,k,h}ds,\]
so that $\eta_t^{x,i,k,h}=(\e^{-t}e_k,0)$. Thus, $\eta^{x,i,k,h}$ admits a continuous modification (notably at $h=0$) and $\partial_{k} (\X^{x,i},\I^{x,i})=(\e^{-t}e_k,0)$ is continuous. Using similar arguments, $\partial_{k}\partial_{l} (\X^{x,i},\I^{x,i})=0$. Gathering those expressions, and since $f^N$ is bounded for every multi-index $N\leq N_1$, it is clear that $P_tf\in\mathscr C^{N_1}$, with, for any $j,k\leq D$,
\begin{align*}
\partial_{j}(P_tf)(x,i)&=\E\left[\partial_{j}f(\X^{x,i}_t,\I^{x,i}_t)\right]\e^{-t},\\
\partial_{j}\partial_{k}(P_tf)(x,i)&=\E\left[\partial_{j}\partial_{k}f(\X^{x,i}_t,\I^{x,i}_t)\right]\e^{-2t}.
\end{align*}
Hence, for any $j\leq N_1,\|(P_tf)^{(j)}\|_\infty\leq\|f^{(j)}\|_\infty$. Finally, for any $n\geq1,|x_n|\leq1$, so that
\[\sup_{n\geq1}\chi_{N_1}(x_n,i_n)=\sup_{n\geq1}\sum_{k=0}^2|x_n|^k\leq 3.\]
Hence, we can apply \cite[Theorems~2.6 and 2.8.ii)]{BBC17} with $N_1=N_2=d_1=d=(2,\dots,2,0),C_T=1,M_2=3,$ to obtain the existence and the announced properties of $\pi$ as well as
\[\lim_{n\to+\infty}d_{\mathscr F}\left(\mathscr L(x_n,i_n),\pi\right)=0.\]

Moreover, following \cite[Remark~2.5]{BBC17},
\[\lambda(\gamma,\epsilon)=-\limsup_{n\to+\infty}\frac{\log(\gamma\vee\epsilon_n)}{\sum_{k=1}^n\gamma_k}\leq \frac\theta A.\]
Finally, using Proposition~\ref{proposition:ErgodicityEZZ} together with \cite[Theorem~2.8.ii)]{BBC17} entails \eqref{eq:ProofThmNonStandard1}. Recall the compactness of $E$, then we can apply \cite[Theorem~2.12]{BBC17} and achieve the proof.
\end{proof}

\subsection{ODE and SDE methods in the standard setting}
\label{sec:StandardOccupation}
In the present section, we successively provide proofs for Theorems~\ref{thm:StandardAS} and \ref{thm:Standard}. We shall prove the former with a method involving an asymptotic pseudotrajectory for some interpolated process, similarly to Section~\ref{sec:NonStandard} and \cite{BC15}. On the contrary, the fluctuations obtained for $(x_n)_{n\geq1}$ in Theorem~\ref{thm:Standard} are obtained through a more classic result for stochastic algorithms, namely the SDE method developed in \cite{Duf96} (see also \cite{KY03}).

\begin{proof}[Proof of Theorem~\ref{thm:StandardAS}]
In the following, we mimic the proof of \cite[Lemma~2.4]{BC15} (see also \cite{MP87,Ben97}). Indeed, for any $n\geq1$, \eqref{eq:DefXAlgoSto} writes
\[x_{n+1}=x_n+\gamma_{n+1}(\nu-x_n)+\gamma_{n+1}(e_{i_{n+1}}-\nu).\]
Let the sequence $(\tau_n)_{n\geq0}$ and the function $m$ be as in \eqref{eq:DefTauM}, and define the interpolated process
\[\widehat{X}_{\tau_n +s}0 = x_n + s \frac{x_{n+1}-x_n}{\tau_{n+1} - \tau_n},\]
for all $s\in [0, \gamma_{n+1})$ and $n\geq0$. We will show that $\widehat{X}$ is an asymptotic pseudotrajectory (and a $\ell-$pseudotrajectory) for the flow $\Phi(x,t) = \nu+e^{-t}(x- \nu)$ associated to the ODE $\partial_t\Phi(x,t)=\nu-\Phi(x,t)$. From \cite[Proposition~4.1]{Ben99} it suffices to show that, for all $T>0,$
\begin{equation}
\label{eq:cvdelta1}
\lim_{t \to+\infty} \Delta(t,T)=0, \quad \text{with }\Delta(t,T) = \sup_{0 \leq h \leq T} \left| \sum_{k=m(t)}^{m(t+h)} \gamma_k(e_{i_{k+1}}-\nu) \right|, 
\end{equation}
and
\begin{equation}
\label{eq:cvdelta2}
\lim_{t \to+\infty} \frac{\log(\Delta(t,T))}{t}\leq - \ell.
\end{equation}
Consider $h$ defined in \eqref{eq:DefH}. Then,
\begin{align}
\gamma_{n+1}\left(e_{i_{n+1}}-\nu\right)
&=\gamma_{n+1}\sum_{j=1}^D q(i_{n+1},j) \left(h_{i_{n+1}}- h_j\right)\notag\\
&=\frac{\gamma_{n+1}}{p_n}  \sum_{j=1}^D \left(p_nq(i_n,j) -q_n(i_n,j)\right)h_{i_{n+1}} + \frac{\gamma_{n+1}}{p_n}\left(h_{i_{n+1}}-\E\left[h_{i_{n+1}}|i_{n}\right]\right) \notag\\
&\quad+\frac{\gamma_{n+1}}{p_n}\sum_{j=1}^D\left(q_n(i_n,j)-p_nq(i_n,j)\right)h_j \notag \\
%&=\frac{\gamma_{n+1}}{p_n}\left(h_{i_{n+1}}-\E\left[h_{i_{n+1}}|i_{n}\right]\right)+\gamma_{n+1}\sum_{j=1}^D\left(q_n(i_n,j)-p_nq(i_n,j)\right)\frac{h_j}{p_n}\notag\\
&\quad+\left(\gamma_{n+1}\sum_{j=1}^Dq(i_n,j)h_j-\gamma_n\sum_{j=1}^Dq(i_n,j)h_j\right) \notag\\
&\quad+\left(\gamma_n\sum_{j=1}^Dq(i_n,j)h_j-\gamma_{n+1}\sum_{j=1}^Dq(i_{n+1},j)h_j\right).
\label{eq:proofOccupation1}
\end{align}
We shall bound each term of the sum \eqref{eq:proofOccupation1} separately. We easily have
\[\left|\gamma_{n+1}\sum_{j=1}^D\left(q_n(i_n,j)-p_nq(i_n,j)\right)\frac{h_j}{p_n}\right|=\left|\gamma_{n+1}\sum_{j=1}^D r_n (i_n,j)h_j\right| \leq \Vert h \Vert_1 R_n \gamma_{n+1}\]
and
$$
\left|\gamma_{n+1}\sum_{j=1}^D  \left(p_nq(i_n,j) -q_n(i_n,j)\right)h_{i_{n+1}} \right|=\left|\gamma_{n+1} h_{i_{n+1}} \sum_{j=1}^D r_n(i_{n+1},j) \right| \leq \Vert h \Vert_1 R_n\gamma_{n+1},  
$$
where $\Vert h \Vert_1 = \sup_j \sum_i |h_{i,j}|$ and $ R_n= \sup_i \sum_j |r_n(i,j)|.$
Also, for some constant $C>0$,
\[\left|\gamma_{n+1}\sum_{j=1}^Dq(i_n,j)h_j-\gamma_n\sum_{j=1}^Dq(i_n,j)h_j\right|\leq C(\gamma_n-\gamma_{n+1}).\]
Note that $(\gamma_n\sum_{j=1}^Dq(i_n,j)h_j-\gamma_{n+1}\sum_{j=1}^Dq(i_{n+1},j)h_j)$ is the main term of a telescoping series. It remains to bound the norm of the sum of $\gamma_{n+1}p_n^{-1} \left(h_{i_{n+1}}-\E[h_{i_{n+1}}|i_{n+1}]\right)$. For all $m,n\geq 1$ and $l=1,...,D$, set 
$$
M_{m,n}(l)=\sum_{k=m}^n\frac{\gamma_{k+1}}{p_k} \left(h_{l,i_{k+1}}-\E\left[h_{l,i_{k+1}}|i_{k+1}\right]\right) .
$$
The sequence $(M_{m,n}(l))_{m\geq n}$ is a martingale and 
\begin{align*}
\E\left[M_{m,n}(l)M_{m,n}(c) \right]
=\sum_{k=m}^n\frac{\gamma_{k+1}^2}{p_k^2}\E\left[\sum_{j=1}^D q_k(i_k,j) (h_{l,j} - \mathbb{E}[h_{l,i_{k+1}} | i_k])(h_{c,j} - \mathbb{E}[h_{c,i_{k+1}} | i_k]) \right].
\end{align*}
Moreover, as 
\begin{align*}
&\E\left[\sum_{j=1}^D q_k(i_k,j) (h_{l,j} - \mathbb{E}[h_{l,i_{k+1}} | i_k])(h_{c,j} - \mathbb{E}[h_{c,i_{k+1}} | i_k]) \right]\\
=&\E\left[\sum_{j=1}^D q_k(i_k,j) (h_{l,j} h_{c,j} - \mathbb{E}[h_{l,i_{k+1}} | i_k] \mathbb{E}[h_{c,i_{k+1}} | i_k]) \right]\\
=&p_k \E\left[\sum_{j\neq i_k} q(i_k,j) h_{l,j} h_{c,j} \right] + \E\left[ \left( 1 - p_k \sum_{j\neq i_k} q(i_k,j) \right) h_{l,i_k} h_{c,i_k}  \right]\\
&- \E\left[ \left(\sum_{j=1}^D q_k(i_k,j) h_{l,j} \right) \left(\sum_{j=1}^D q_k(i_k,j) h_{c,j} \right) \right] +o(p_k) \\
%=&p_k \E\left[\sum_{j\neq i_k} q(i_k,j) h_{j,l} h_{j,c} \right] + \E\left[ \left( 1 - p_k \sum_{j\neq i_k} q(i_k,j) \right) h_{i_k,l} h_{i_k,c}  \right]- \E\left[ p_k^2 \left(\sum_{j\neq i_k} q(i_k,j) h_{j,l} \right) \left(\sum_{j\neq i_k} q(i_k,j) h_{j,c} \right) \right. \\
%&- p_k \left( 1 - p_k \sum_{j\neq i_k} q(i_k,j) \right) h_{i_k,l} \sum_{j\neq i_k} q(i_k,j) h_{j,c} - p_k \left( 1 - p_k \sum_{j\neq i_k} q(i_k,j) \right) h_{i_k,c} \sum_{j\neq i_k} q(i_k,j) h_{j,l}\\
%&- \left. \left( 1 - p_k \sum_{j\neq i_k} q(i_k,j) \right)^2 h_{i_k,c} h_{i_k, l}   \right] +o(p_k) \\
=&p_k \E\left[\sum_{j\neq i_k} q(i_k,j) (h_{l,j} - h_{l,i_k})(h_{c,j} - h_{c,i_k)} ) \right] \\
&+ \E\left[ p_k^2  \left(\sum_{j\neq i_k} q(i_k,j) (h_{l,i_k} - h_{l,j})\right) \left( \sum_{j\neq i_k} q(i_k,j) (h_{c,j} - h_{c,i_k}) \right) \right] +o(p_k), 
\end{align*}
by Theorem~\ref{thm:StandardAS}, we obtain
\begin{align}
&\E\left[\sum_{j=1}^D q_k(i_k,j) (h_{l,j} - \mathbb{E}[h_{l,i_{k+1}} | i_k])(h_{c,j} - \mathbb{E}[h_{c,i_{k+1}} | i_k]) \right]\notag\\
&=p_k \sum_{i=1}^D \nu_i \left[\sum_{j=1}^D q(i,j) (h_{l,j} - h_{l,i})(h_{c,j} - h_{c,i)} ) \right]\notag\\
&\quad- p_k^2 \sum_{i=1}^D \nu_i \left(\nu_l-\indic_{i=l}\right) \left( \nu_c-\indic_{i=c} \right)  +o(p_k).\label{eq:thebrackets}
\end{align}
As a consequence of \eqref{eq:thebrackets}, there exists some constant $C>0$ such that
\[\sum_{l=1}^D \E\left[M_{m,n}(l)^2\right]\leq C \sum_{k=m}^n \frac{\gamma_{k+1}^2}{p_k}.\]
By Doob's inequality and Assumption~\ref{assumption:Standard}, it follows that, for every $k\geq 0$,
\begin{align*}
\mathbb{E}\left[ \sup_{0\leq h \leq T} |M_{m(kT), m(kT +h)}|\right] 
&\leq  2 \sum_{l=1}^D \mathbb{E}\left[|M_{m(kT), m((k+1)T)}(l)|^2\right]\leq 2C \sum_{j = m(kT)}^{m((k+1)T)} \gamma_{j+1}\frac{\gamma_{j+1}}{p_j}\\
&\leq 2CT \sup_{j\geq m(kT)} \frac{\gamma_{j+1}}{p_{j}},
\end{align*}
which implies that $\lim_{k \to+\infty}  \sup_{0\leq h \leq T} |M_{m(kT), m(kT +h)}|= 0$ and then $\lim_{k \to+\infty} \Delta(kT,T) =0$ in probability. By the triangle inequality and \cite[Proposition~4.1]{Ben99}, \eqref{eq:cvdelta1} holds.

Under the assumption that $\sum_{n=1}^\infty\gamma_{n+1}^2p_n^{-1} <+ \infty,$
\begin{align*}
\mathbb{E}\left[ \sum_{k\geq 0} \sup_{0\leq h \leq T} |M_{m(kT), m(kT +h)}|\right] 
&\leq \sum_{k\geq 0} 2 \sum_{l=1}^D \mathbb{E}\left[|M_{m(kT), m((k+1)T)}(l)|^2\right]
\leq 2C \sum_{k \geq m(T)} \frac{\gamma_{k+1}^2}{p_k} < + \infty,
\end{align*}
which implies $\lim_{k \to+\infty}  \sup_{0\leq h \leq T} |M_{m(kT), m(kT +h)}|= 0$ a.s. Then, $\lim_{k \to+\infty} \Delta(kT,T) =0$ a.s. and $\lim_{t \to+\infty} \Delta(t,T) =0$ since
$$
\Delta(t,T) \leq 2\Delta( \lfloor t/T \rfloor T, T) + \Delta( (\lfloor t/T \rfloor +1) T, T).
$$

In order to obtain a $\ell$-pseudotrajectory, use Markov's and Doob's inequalities so that
$$
\mathbb{P}\left( \sup_{0\leq h \leq T} |M_{m(kT), m(kT +h)}| \geq e^{-kT \alpha} \right) \leq e^{k T \alpha } 2C \sum_{j=m(kT)}^{m\left((k+1)T\right)} \gamma_{j+1}\frac{\gamma_{j+1}}{p_j} \leq 2C(T+1)  e^{k T \alpha } \sup_{j\geq m(kT)} \frac{\gamma_{j+1}}{p_j} .
$$
Now, for all $\varepsilon>0$ and $k$ large enough,
\[\sup_{j\geq m(kT)} \frac{\gamma_{j+1}}{p_j} \leq \exp\left(- (\lambda(\gamma, \gamma/p))- \varepsilon) \tau_{m(kT)}\right) \leq \exp\left(- (\lambda(\gamma, \gamma/p))- \varepsilon) kT\right),\]
where $\lambda(\gamma,\gamma/p)$ is defined in \eqref{eq:DefLambdaGammaEpsilon}. Hence,
$$
\mathbb{P}\left( \sup_{0\leq h \leq T} |M_{m(kT), m(kT +h)}| \geq e^{-kT \alpha} \right) \leq 2C(T+1)  \exp\left(k T (\alpha - \lambda(\gamma, \gamma/p) + \varepsilon) \right),
$$
and by the Borel-Cantelli lemma, we have
$$
\limsup_{k \to+\infty}\frac{1}{k} \sup_{0\leq h \leq T} |M_{m(kT), m(kT +h)}| \leq - \lambda(\gamma, \gamma/p)\text{ a.s.}
$$
Then, bounding all the other terms of \eqref{eq:proofOccupation1}, we find
$$
\lim_{t \to+\infty} \frac{\log(\Delta(t,T)}{t}\leq - \ell
$$
with
$$
\ell = \min\Big(\lambda(\gamma, \gamma/p), \lambda(\gamma,\gamma), \lambda(\gamma, R )\Big)= \lambda(\gamma, \gamma/p)\wedge \lambda(\gamma, R).
$$
Since the flow $\Phi$ converges to $\nu$ exponentially fast at rate $1$, use \cite[Theorem~6.9 and Lemma~8.7]{Ben99} to achieve the proof.
\end{proof}

\begin{proof}[Proof of Theorem~\ref{thm:Standard}]
We have
$$
y_{n+1} = y_n + y_n \left( \frac{\alpha_{n+1}}{\alpha_n}(1-\gamma_{n+1}) - 1 \right) + \gamma_{n+1} \alpha_{n+1} (e_{i_{n+1}} - \nu).
$$
Recall \eqref{eq:proofOccupation1}, so that 
$$
\gamma_{n+1} (e_{i_{n+1}} - \nu) = \frac{\gamma_{n+1}}{p_n} (h_{i_{n+1}} - \mathbb{E}[h_{i_{n+1}}|i_n]) + b_n,
$$
with a remainder term $b_n$ converging to 0. Now, we want to use \cite[Th\'{e}or\`{e}me~4.II.4]{Duf96}. In our setting, its notation reads
$$
y_{n+1} = y_n + \widehat{\gamma}_n \left( h(y_n) + \widehat{r}_{n+1} \right) + \sqrt{\widehat{\gamma}_n} \epsilon_{n+1},
$$
with
$$
\epsilon_{n+1}= \left(\frac{1+\Upsilon}{2} \right)^{-1/2}  \frac{1}{\sqrt{p_n}} (h_{i_{n+1}} - \mathbb{E}[h_{i_{n+1}}|i_n]) , \quad  \widehat{\gamma}_{n+1}= \left(\frac{1+\Upsilon}{2} \right) \frac{\gamma_{n+1}^2 \alpha_{n+1}^2}{p_n},\quad  h:z\mapsto -  z,
$$
and
\begin{align*}
\widehat{r}_{n+1}
&= y_n \frac{1}{\widehat{\gamma}_{n+1}} \left( \frac{\alpha_{n+1}}{\alpha_n}(1-\gamma_{n+1}) - 1  +  \widehat{\gamma}_{n+1} \left(\frac{1+\Upsilon}{2} \right)  \right)\\
&\quad + \frac{\alpha_{n+1}}{ \widehat{\gamma}_{n+1}} \frac{\gamma_{n+1}}{p_n}  \sum_{j=1}^D r_n(i_n,j)(h_{i_{n+1}} -h_j) \notag \\
&\quad+\frac{\alpha_{n+1}}{ \widehat{\gamma}_{n+1}} \gamma_{n+1} \left(\sum_{j=1}^Dq(i_n,j)h_j-\sum_{j=1}^Dq(i_{n+1},j)h_j\right).
\end{align*}
Then, by \eqref{eq:thebrackets} and similar computations,
$$
\mathbb{E}[\epsilon_n | \mathcal{F}_n] = 0 \quad \mathbb{E}[\epsilon_n \epsilon_n^\top | \mathcal{F}_n] =  \Sigma^{(p,\Upsilon)} +o(p_n), \quad \sup_{n \geq 1} \mathbb{E} \left[ |\epsilon_n |^q \right] < + \infty, \ q>2,
$$
where $\Sigma^{(p,\Upsilon)}$ is defined in \eqref{eq:DefSigma}. Classically, we should prove that $\lim_{n\to+\infty}\| \widehat{r}_n \|=0$, in order to work in the framework of \cite[Hypoth\`ese~H4-4]{Duf96}, which is quite difficult. Nevertheless, rather than checking that $\lim_{n\to+\infty}\| \widehat{r}_n \| = 0$ it is sufficient\footnote{This assertion can be easily checked at the end of \cite[p.156]{Duf96}, whose proof is based on usual arguments on diffusion approximation, such as \cite{EK86}. The decomposition \eqref{eq:proofThmStandard1} is often assumed in more recent generalizations, see for instance \cite{For15}. Note that we cannot use directly \cite{For15}, which besides does not provide functional convergence.} to prove that
\begin{equation}
\widehat{r}_n= \widehat{r}^{(1)}_n + \widehat{r}^2_n,\quad\lim_{n\to+\infty}\widehat{r}^{(1)}_n=0,\quad \lim_{n\to+\infty}\mathbb{E} \left[\sup_{0 \leq s \leq T} \left|\sum_{n=m(t)}^{m(t+s)} \widehat{\gamma}_{n+1} \widehat{r}^{(2)}_{n+1} \right| \right] = 0,\label{eq:proofThmStandard1}
\end{equation}
for any $T>0$, where $m(t)$ is defined in \eqref{eq:DefTauM}. Then, let
\begin{align*}
\widehat{r}^{(1)}_{n+1}
&= y_n \frac{1}{\widehat{\gamma}_{n+1}} \left( \frac{\alpha_{n+1}}{\alpha_n}(1-\gamma_{n+1}) - 1  +  \widehat{\gamma}_{n+1} \left(\frac{1+\Upsilon}{2} \right) \right)
+ \frac{\alpha_{n+1}}{ \widehat{\gamma}_{n+1}} \frac{\gamma_{n+1}}{p_n}  \sum_{j=1}^D r_n(i_n,j)(h_{i_{n+1}} -h_j),\\
\widehat{r}^{(2)}_{n+1}
&= \frac{\alpha_{n+1}}{ \widehat{\gamma}_{n+1}} \gamma_{n+1} \left(\sum_{j=1}^Dq(i_n,j)h_j-\sum_{j=1}^Dq(i_{n+1},j)h_j\right).
\end{align*}
The sequence $(\widehat{r}^{(1)}_n)_{n\geq1}$ goes to $0$ a.s. and in $L^1$ straightforwardly under our assumptions. Furthermore
\begin{align}
\widehat{\gamma}_{n+1} \widehat{r}^{(2)}_{n+1}
&= \alpha_{n+1} \gamma_{n+1} \sum_{j=1}^Dq(i_n,j)h_j-\alpha_{n+2} \gamma_{n+2} \sum_{j=1}^Dq(i_{n+1},j)h_j\notag\\
&\quad + (\alpha_{n+2} \gamma_{n+2} - \alpha_{n+1} \gamma_{n+1}) \sum_{j=1}^Dq(i_{n+1},j)h_j.
\label{eq:proofThmStandard2}
\end{align}
The first line of \eqref{eq:proofThmStandard2} is a telescoping series and is bounded by $\alpha_n \gamma_{n+1}$ which goes to $0$. The second line of \eqref{eq:proofThmStandard2} is bounded by,
\begin{equation}
C \sum_{n= m(t)}^{m(t+T)} \left| \alpha_{n+2} \gamma_{n+1} - \alpha_{n+1} \gamma_{n} \right|,
\label{eq:proofThmStandard3}
\end{equation}
for some $C>0$. Since \eqref{eq:proofThmStandard2} is a telescoping series as well, and goes to $0$, we established the announced decomposition \eqref{eq:proofThmStandard1}. As a conclusion, the diffusive limit $(\Y_t)_{t\geq0}$ is the solution of \eqref{eq:SdeOU}, which trivially  admits $V:z\mapsto z$ as a Lyapunov function, as required in  \cite[Hypoth\`ese~H4-3]{Duf96}. The only use of an assumption on the eigenelements of $\Sigma^{(p,\Upsilon)}$ would be to guaranty the existence, uniqueness of and convergence to an invariant distribution for $\Y$, which was already proved in Proposition~\ref{proposition:ErgodicityOU}.
\end{proof}

\begin{acknowledgements}
	Both authors acknowledge financial support from the ANR PIECE (ANR-12-JS01-0006-01) and the Chaire Mod\'elisation Math\'ematique et Biodiversit\'e.
\end{acknowledgements}

\bibliography{Biblio}

\end{document}